\documentclass[a4paper, 12pt]{amsart}
\usepackage{standalone}
\usepackage[T2A]{fontenc}
\usepackage[utf8]{inputenc}
\usepackage[english, russian]{babel}
\usepackage{amsmath,amssymb,amsthm,url}
\usepackage{dsfont,textcomp,graphicx,overpic,wrapfig}
\usepackage[usenames,dvipsnames]{color}
\graphicspath{{figs/}}

\newcommand{\algor}[1]{}
\newcommand{\lktg}[1]{}
\newcommand{\invadraw}[1]{}

\newcommand{\aronly}[1]{#1}
\newcommand{\jonly}[1]{}
\newcommand{\compile}[1]{#1}
   
\makeatletter

\@addtoreset{figure}{section}
\makeatother

\usepackage{hyperref}
\hypersetup{
   colorlinks   = true, 
   urlcolor     = blue, 
   linkcolor    = blue, 
   citecolor   = red 
}

\input{xy}\xyoption{all}

\usepackage{tikz}
\usetikzlibrary{positioning, calc, intersections, through, arrows}
\usetikzlibrary{graphs}
\usetikzlibrary{graphs.standard}
\usetikzlibrary{patterns}

\renewcommand{\ge}{\geqslant}

\renewcommand{\le}{\leqslant}

\renewcommand{\emptyset}{\varnothing}

\def\sgn{\mathop{\fam0 sgn}}

\def\t{\widetilde}
\def\R{{\mathbb R}} \def\Z{{\mathbb Z}}  \def\Q{\Bbb Q}

\def\tr{\mathop{\fam0 tr}}
\renewcommand\tr{\qopname\relax o{tr}}
\def\cy{\mathop{\fam0 cy}}
\renewcommand\cy{\qopname\relax o{cy}}

\def\triod{\mathop{\fam0 triod}}
\renewcommand\triod{\qopname\relax o{triod}}
\def\adiag{\mathop{\fam0 adiag}}
\renewcommand\adiag{\qopname\relax o{adiag}}

\theoremstyle{plain}
\newtheorem{theorem}{Теорема}[section]
    \newtheorem{lemma}[theorem]{Лемма}
    
    \newtheorem{proposition}[theorem]{Утверждение}
    \newtheorem{conjecture}[theorem]{Гипотеза}
    \newtheorem{problem}[theorem]{Проблема}
\newtheorem{example}[theorem]{Пример}
\newtheorem*{example*}{Пример}

\newtheoremstyle{mydefinition}
  {\medskipamount}
  {\medskipamount}
  {\normalfont}
  {\parindent}
  {\bfseries}
  {.}
  { }
  {}

\theoremstyle{mydefinition}

\newtheorem{remark}[theorem]{Замечание}

\usepackage{pdflscape}

\begin{document}
 
\newcommand{\mytitle}{Инварианты почти вложений графов в плоскость}
 
\title{\mytitle} 
 
\author{Э. Алкин, А. Мирошников, А. Скопенков}

\address{\emph{Э. Алкин:} Московский Центр непрерывного математического образования.
\newline
\emph{А. Мирошников:} Московский физико-технический институт.
\newline
\emph{А. Скопенков:} Московский Центр непрерывного математического образования, \url{https://users.mccme.ru/skopenko/}.}

\thanks{
Исследование выполнено за счет гранта Российского научного фонда № 25-21-00685, \url{https://rscf.ru/project/25-21-00685/}.
Мы благодарны Д. Акимову, М. Дидину, Д. Мусатову, О. Стырту и особенно Е. Дженжер (Бордачевой), Т. Гараеву и О. Никитенко за полезные обсуждения.
} 
   
\date{}

\begin{abstract}



  

Изображение графа на плоскости называется \emph{почти вложением}, если образы любых двух несмежных симплексов (т.~е. вершин или ребер) не пересекаются.
Почти вложения (точнее, их многомерные обобщения) естественно возникают в комбинаторной геометрии, топологической комбинаторике, а также при изучении вложений гиперграфов.
Мы доказываем некоторые соотношения между инвариантами почти вложений.
Мы показываем связь некоторых из этих соотношений с \emph{гомологиями взрезанного квадрата} графа.
Мы строим почти вложения, реализующие некоторые значения этих инвариантов.

В этом обзоре доступно (для неспециалистов в этой области, в частности, студентов) излагаются некоторые идеи алгебраической и геометрической топологии.
Все необходимые определения напоминаются.
Несмотря на элементарность, обзор мотивирован передним краем науки; имеется несколько гипотез и нерешенных проблем.
\end{abstract}

\maketitle
\tableofcontents

\section{Введение}\label{s:intr}

\emph{О стиле этого текста.} 
Основные идеи представлены на <<олимпиадных>> примерах: на простейших частных случаях, свободных от технических деталей, и со сведением к необходимому минимуму научного языка.
Благодаря этому и большая часть текста доступна для неспециалистов в топологии (в частности, для специалистов по компьютерной науке), и удается быстро добраться до интересных результатов. 

Все необходимые определения (отображения графа в плоскость, числа оборотов замкнутой ломаной вокруг точки и т.д.) напомнены в \S\ref{s:wind}, \S\ref{s:wingra}.
Приводимые там определения более простые, чем во многих учебниках.

\medskip
Изображения без самопересечений графов на плоскости (т.~е. вложения или плоские графы) активно изучаются.
Также интересны изображения графов, имеющие <<умеренные>> самопересечения, например, определяемые чуть ниже \emph{почти вложения}. 

\begin{figure}[ht]\centering
    \compile{
        \includegraphics[scale=1.5]{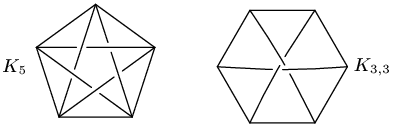} 
    }
    \caption{Непланарные графы $K_5$ и $K_{3,3}$}
    \label{planar}
\end{figure}

Одно из доказательств непланарности (т.~е. невложимости в плоскость) графа $K_5$ или $K_{3,3}$ обобщается до следующего результата.

\begin{theorem}[Ханани--Татт; ван Кампен--Флорес]\label{t:hatuvk}
Пусть граф $K$ --- это $K_5$ или $K_{3,3}$. 
Для любого отображения $K\to\R^2$ существуют два несмежных ребра графа $K,$ образы которых пересекаются.
\end{theorem}

Эта теорема вытекает \cite[\S1.4]{Sk18} из ее количественной версии (теорема~\ref{t:hatuvkfl}).

\begin{figure}[ht]\centering
\compile{
    \includegraphics[width=4.6cm]{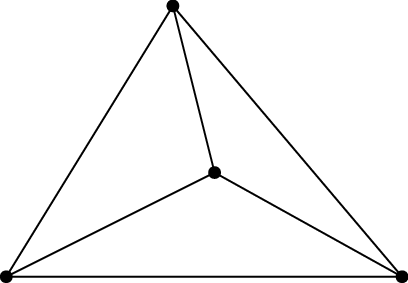}
    \qquad
    \includegraphics[width=4.6cm]{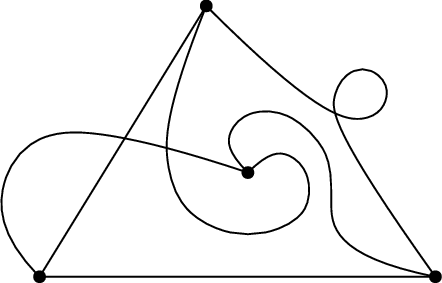}
    \qquad
    \includegraphics[width=3cm]{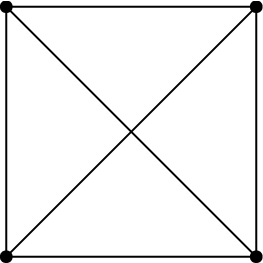}
}
    \caption{Вложение, почти вложение и отображение (изображение), которое не является почти вложением}
    \label{f:k4}
\end{figure}
 
\begin{figure}[ht]\centering
\compile{
    \includegraphics[scale=0.7]{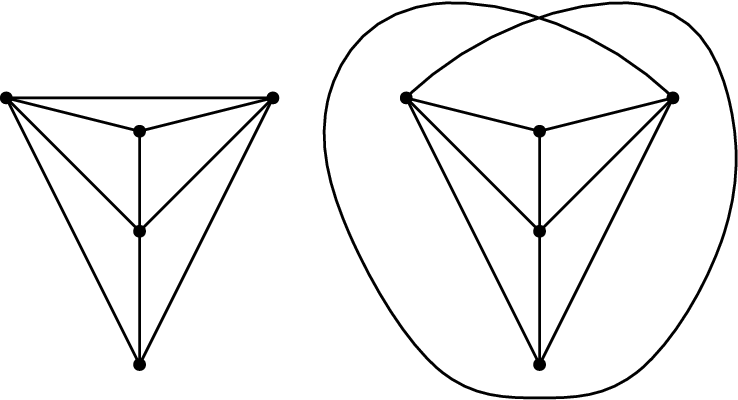}
    }
    \caption{Вложение и почти вложение графа $K_5$ без ребра}
    \label{k5}
\end{figure}  

Отображение $f:K\to\R^2$ произвольного графа $K$ называется \textbf{почти вложением}, если $f(\alpha)\cap f(\beta) = \varnothing$ для любых двух несмежных симплексов (т.~е. вершин и ребер) $\alpha,\beta\subset K$. Т.~е. если

(i) образы несмежных ребер не пересекаются,

(ii) образ любой вершины не лежит на образе никакого ребра, несмежного с этой вершиной, 

(iii) образы различных вершин различны.

Например, 

$\bullet$ почти вложение графа $K_3$ в плоскость определяется тройкой точек $A,B,C$ на плоскости вместе с тройкой ломаных $A\ldots B$, $B\ldots C$, $C\ldots A$, каждая из которых проходит ровно через две точки из $\{A,B,C\}$;

$\bullet$ почти вложение графа $K_4$ в плоскость определяется четверкой точек $A,B,C,D$ на плоскости вместе с тремя парами $\{A\ldots B,C\ldots D\}$, $\{A\ldots C,B\ldots D\}$, $\{A\ldots D,B\ldots C\}$ ломаных, в каждой паре из которых ломаные не пересекаются;   

$\bullet$ почти вложений графов $K_5$ и $K_{3,3}$ в плоскость не существует (теорема~\ref{t:hatuvk}).

Более подробная мотивировка понятия почти вложения приведена в замечании~\ref{r:ae_graph_mot}.

\begin{remark}[обсуждение определения почти вложения]\label{ss:ae_def_discussion}
Этот текст касается в первую очередь \emph{инвариантов} почти вложений, а не проблем \emph{существования} почти вложений. 
Поэтому мы сохраняем в определении свойства (ii, iii), выполнения которых можно добиться достаточно малым шевелением отображения, с сохранением свойства (i).  
\emph{Если граф допускает почти вложение в плоскость, то граф планарен} (это следует из теоремы \ref{t:hatuvk} и теоремы Куратовского \cite[Theorem XI.60]{Tu84}).
Однако существуют значения инвариантов, реализуемые некоторыми почти вложениями, но не реализуемые никакими вложениями (это следует из примеров \ref{e:k3}.ac, а также из второго предложения после формулировки теоремы~\ref{t:radonae} вместе с примером~\ref{e:k3}.b и теоремой~\ref{t:k4}).
\end{remark}

{\bf Описание содержания обзора.}

\emph{Мы изучаем} целочисленные инварианты почти вложений: оборотное (\S\ref{s:wingra}), циклическое и триодическое числа\footnote{Видимо, определения триодического и циклического чисел не встречаются в литературе. Но эти определения просты и естественны, поэтому их следует считать известными.
} (\S\ref{s:gawh}).
Здесь инварианты "--- просто числа, сопоставленные почти вложению (мы не используем понятие \emph{почти изотопии}, относительно которой эти числа инвариантны, вне замечания \ref{r:ae_graph_mot}.a и \S\ref{s:class}).

\emph{Мы приводим} некоторые \emph{соотношения} между значениями этих инвариантов. 

Некоторые из них связаны с \emph{теоремой~\ref{t:borsuk} Борсука--Улама}.
Для графа $K_4$ такое соотношение (теорема~\ref{t:radonae}) --- эквивалентная версия знаменитой топологической теоремы~\ref{t:radon} Радона. 
Для графа, полученного из $K_5$ удалением ребра, такое соотношение (теорема~\ref{t:k5-e}.а) --- эквивалентная версия теоремы~\ref{t:hatuvk} (точнее, ее количественной версии "--- знаменитой теоремы~\ref{t:hatuvkfl} Ханани--Татта--ван Кампена--Флореса).
Более сильное соотношение для графа $K_5$ без ребра (теорема~\ref{t:k5-e}.b) "--- недавний нетривиальный результат Тимура Гараева \cite{Ga23}.
Он уже не следует из теоремы Борсука--Улама или ее аналога, а \emph{доказывается геометрически}.
Утверждения о нечетности циклического (\ref{pr:off}) и триодического (\ref{pr:triod}) чисел "--- простые следствия теоремы Борсука--Улама (см. вывод в замечании~\ref{r:cyc}.c; мы также приводим прямые элементарные доказательства). 

Другие доказываемые соотношения 
связаны с гомологиями \emph{конфигурационного  пространства} $\widetilde K$ 
размещений из 
графа~$K$ по две точки (см. подробное определение \aronly{и обсуждение }в замечании~\ref{r:delpro-def}).
Это соотношения между инвариантами почти вложений графа $K_4$ (утверждения~\ref{p:wu-conj0} и~\ref{p:wu-conj}), графа $K_5$ без ребра (утверждение~\ref{p:k4-k5}), и графа $K_{3,3}$ без ребра (утверждение~\ref{p:k4-k33}).

\emph{Мы рассматриваем инварианты, обобщающие} оборотные, циклические и триодические числа.
Эти инварианты, называемые $C$-числами Ву, параметризованы элементами~$C$ \emph{группы гомологий} $H_1(\widetilde K;\Z)$ 
(см. определение в абзаце перед проблемой~\ref{pm:Wu}).
Оказывается, что \emph{любое $C$-число Ву выражается через оборотные, циклические и триодические числа} (теоремы~\ref{t:poly-Wu} и~\ref{t:Wu}).


\emph{Мы приводим построения} почти вложений, реализующих некоторые значения изучаемых инвариантов. 
Один из примеров (теорема~\ref{t:k4}) показывает, что нет ограничений на значения оборотных чисел для графа $K_4$, кроме теоремы~\ref{t:radonae}. 
(Более точно, в предыдущем предложении <<оборотных чисел>> нужно заменить на <<оборотных чисел вида $w_f(j)$>>; см. определение перед утверждением~\ref{p:autom} и обсуждение перед проблемой~\ref{pm:wind}.)
Это недавний результат Э. Алкина и А. Мирошникова \cite{AM25}.
Другой пример (\ref{e:k33-strong}) показывает, что нет ограничений на значения оборотных чисел, упомянутых в теореме~\ref{t:k33-odd}, для графа $K_{3,3}$ без ребра, кроме теоремы~\ref{t:k33-odd}.
Остальные примеры просты и наглядны (примеры~\ref{p:off},~\ref{p:triod}; рисунки~\ref{f:k4-pm3},~\ref{f:k23-conj}).










\emph{Мы выдвигаем нерешенную проблему~\ref{pm:wind} и ее обобщение~\ref{pm:Wu}} об описании значений инвариантов, реализуемых почти вложениями.
Мы выделяем ее частный случай о реализации почти вложениями значений оборотных чисел для графа, полученного из $K_{5}$ удалением ребра (гипотеза~\ref{c:k5-e}), 
а также обсуждаем связанные с ней гипотезы (\ref{r:int_analogs} и~\ref{c:tree}).
Проблемы~\ref{pm:wind}~и~\ref{pm:Wu} для графа $K_4$ решаются теоремами~\ref{t:radonae},~\ref{t:k4} (для проблемы \ref{pm:Wu} нужны еще дополнительные рассуждения после ее формулировки).
Для решения самой проблемы~\ref{pm:Wu} 
могут быть полезны вышеупомянутые теоремы~\ref{t:poly-Wu} и \ref{t:Wu} о выражаемости.

Все упомянутые выше теоремы, утверждения и примеры, кроме теорем~\ref{t:poly-Wu},~\ref{t:Wu}, утверждений~\ref{p:k4-k5},~\ref{p:k4-k33},~\ref{p:wu-conj0},~\ref{p:wu-conj} и примера~\ref{e:k33-strong}, следует считать известными (даже если они не были ранее опубликованы в таком виде).

Несмотря на элементарность, обзор подводит читателя к переднему краю науки, см. замечания~\ref{r:gene},~\ref{r:gap}, статьи \cite{IKN+, Ni22, Ga23, AM25} (графы на плоскости и в пространстве), \cite{Sk18} и 
ссылки в \cite[первый пункт в последнем абзаце части `Some motivation']{SS23}
(гомологии пространств размещений/сочетаний из графа по две точки).

 
\begin{remark}[к мотивировке понятия почти вложения для графов]
\label{r:ae_graph_mot}
(a) Непрерывная деформация в классе почти вложений (\emph{почти изотопия}) изучается для графов в $\R^3$, начиная с \cite{Ta94} (под названиями \emph{слабая гомотопия} или \emph{вершинная гомотопия}). Ссылки на более поздние статьи, в которых изучается это понятие, см. в \cite{FN09}.


(b) Естественное более общее понятие (\emph{AT graph}) изучается, начиная с \cite{KLN}; ссылки на более поздние статьи, в которых оно изучается, см. в \cite{Ky20}.
Многомерный аналог этого общего понятия настолько полезен, что он использовался в \cite[Disjunction Theorem 3.1]{Sk02} без явного выделения термина.



(c) Ясно, что свойство <<быть почти вложением>> сохраняется при достаточно малых шевелениях отображения (в отличие от свойства <<быть вложением>>, аналогичная устойчивость справедлива для свойств <<быть $\Z_2$-вложением>> и <<быть $\Z$-вложением>>; см. определения в замечании~\ref{r:zz2}).
Поэтому, аппроксимируя непрерывное отображение кусочно-линейным (PL), получаем, что для графов в поверхностях 

$\bullet$ топологическое вложение может быть приближено PL почти вложением;

$\bullet$ PL или топологическое почти вложение может быть приближено PL почти вложением общего положения;

$\bullet$ PL почти вложимость эквивалентна топологической почти вложимости.







(d) Почти вложения графов интересны также как частный случай почти вложений гиперграфов в многомерные пространства, см. обсуждение основных результатов в~\S\ref{s:higher}.
\end{remark}

\begin{remark}[алгебраические версии почти вложений]
\label{r:zz2}
При изучении вложений (включая многомерный случай) естественно возникают алгебраические версии почти вложений 
($\Z_2$- и $\Z$-вложения). 

Возьмем отображение $h : K \to M$ общего положения (см. строгое определение в \cite[Definition 1.1.6.a при $k=1$]{Sk24}) графа $K$ в двумерную поверхность $M$. 
Оно называется \emph{$\Z_2$-вложением}, если число $|h \sigma \cap h \tau|$ четно для любой пары $\sigma, \tau$ несмежных ребер.
Для ориентируемой поверхности $M$, оно называется \emph{$\Z$-вложением}, если сумма знаков точек пересечений <<ориентированных ломаных>> $h \sigma$, $h \tau$ равна нулю для любой пары $\sigma, \tau$ несмежных ребер.

Понятие $\Z_2$-вложения (другое название "--- \emph{отображение Ханани--Татта})
появилось в 1930-х годах (см. теорему~\ref{t:hatuvk}) и активно изучается с 2000-х годов, см. обзор~\cite{Sc13} и \cite{SS13, FK19, Bi21}. Оно изучается не только для изображений графов на поверхностях, но и для многомерного случая (см. замечание~\ref{r:ae_complex_mot}.b).

{\bf Теорема.}
{\it (a) Если граф $\Z_2$-вложим в ленту Мебиуса, то этот граф вложим в ленту Мебиуса. \cite{CKP+}, \cite{PSS}.
Если граф $\Z_2$-вложим в тор, то этот граф вложим в тор. \cite{FPS}.

(b) Существует связный граф $\Z_2$-вложимый в сферу с четырьмя ручками, но не вложимый в нее. \cite{FK17}.}

Неизвестно, влечет ли $\Z$-вложимость графа в ориентируемую поверхность вложимость этого графа в нее.

\emph{Для любой поверхности существует полиномиальный (от размера графа) алгоритм, который по графу $K$ проверяет $\Z_2$-вложимость графа $K$ в эту поверхность}.
Это в сущности известно. 
Это следует из того, что свойство графа <<быть $\Z_2$-вложимым в данную поверхность>> сохраняется при переходе к его минору (т.~е. при удалении ребра или стягивании ребра).
Таким образом, по теореме Робертсона--Сеймура о минорах~\cite{RS04} существует конечное число запрещенных миноров, характеризующих такое свойство. 
Также, существует полиномиальный (от размеров графов) алгоритм, проверяющий является ли данный граф минором другого данного графа~\cite{KKR}. 
Поэтому нужный алгоритм получается последовательным запуском алгоритмов проверки наличия каждого запрещенного минора в графе.\footnote{Поскольку мы лишь обосновываем существование такого алгоритма, а не приводим конкретного способа его построения, нам не нужно знать алгоритма, перечисляющего запрещенные миноры. То есть это обоснование неконструктивно. Например, не выписаны запрещенные миноры даже для $\Z_2$-вложимости в тор. Кроме того, полиномиальный алгоритм проверки наличия запрещенного минора в графе может на практике быть даже хуже экспоненциального, т.~е. может быть <<галактическим>>~\cite{GA}.
}

Замечание, аналогичное предыдущему абзацу, справедливо для $\Z$-вложимости. 

\end{remark}

{\bf Соглашения.}
Если к утверждению не приведено ни доказательство, ни ссылка на него, то оно несложно.
Определения важных понятий даны \textbf{жирным шрифтом}, чтобы их было проще найти.
\emph{Замечания} не используются в дальнейшем. 
Рассматриваемые точки, замкнутые ломаные и ломаные расположены на плоскости $\R^2$.

\section{Число оборотов: определение и обсуждение}\label{s:wind}
 
Пусть $O,A,B,A_1,\ldots,A_m$ "--- точки.  
 
Предположим, что $A\ne O$ и $B\ne O$ (но, возможно, $A=B$). 
\emph{Ориентированным (или направленным) углом} $\angle AOB$ называется число  $t\in(-\pi,\pi]$, такое что вектор $\overrightarrow{OB}$ сонаправлен вектору, полученному из $\overrightarrow{OA}$ вращением на угол $t$ против часовой стрелки. 
(Если рассматривать векторы на плоскости как комплексные числа, то можно переписать это условие как $\overrightarrow{OB}\upuparrows e^{it}\overrightarrow{OA}$.)

\textbf{Замкнутой ломаной} называется упорядоченный набор точек (не обязательно различных).\footnote{
\label{fn:sets}Таким образом, замкнутая ломаная (определенная здесь) не является подмножеством плоскости.
Тем не менее, иногда мы работаем с замкнутой ломаной $A_1\ldots A_m$ как с объединением отрезков $A_iA_{i+1}$, например, мы пишем <<ломаная, не проходящая через точку>>.}
В этом тексте мы иногда обозначаем упорядоченный набор $(A_1, \ldots, A_n)$ через $A_1\ldots A_n$ или $A_1, \ldots, A_n$.

{\bf Числом оборотов} $w(l) = w(l,O)$ замкнутой ломаной $l=A_1\ldots A_m$ вокруг не лежащей на ней точки $O$ называется количество оборотов при вращении вектора, начало которого находится в точке $O$, а конец обходит ломаную в заданном направлении.
Строго говоря, 
$$2 \pi \cdot w(l) = 2\pi \cdot w(l,O) := \angle A_1OA_2+\angle A_2OA_3+\ldots+\angle A_{m-1}OA_m+\angle A_mOA_1$$
"--- сумма ориентированных углов.
См. рисунки \ref{f:abco} и \ref{f:wn_examples}.

\begin{figure}[ht] \centering
\compile{
    \includegraphics[scale=0.65]{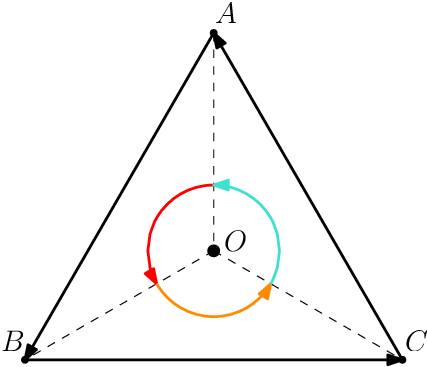}
    }
    \caption{$w(ABC) =\dfrac{1}{2\pi} \left( \angle AOB + \angle BOC + \angle COA \right) = +1 $}
    \label{f:abco}
\end{figure}

\begin{figure}[ht] \centering
\compile{
    \includegraphics[scale=1.0]{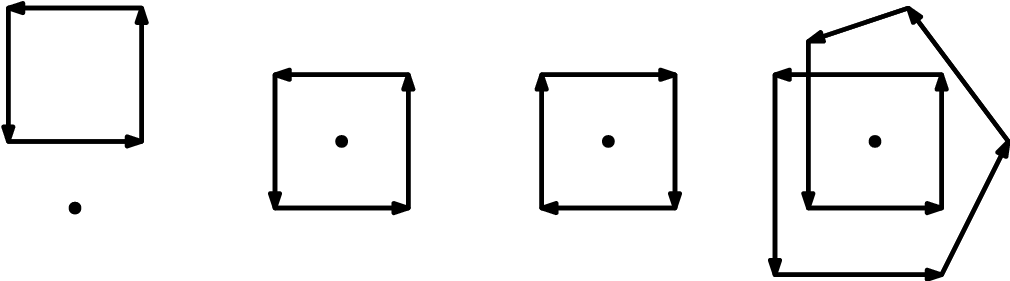}
    }
    \caption{Числа оборотов равны $0,~+1,~-1,~+2$}
    \label{f:wn_examples}
\end{figure}

Пусть $ABC$ "--- треугольник и  $O$ "--- точка внутри него.
Тогда $w(ABCABC) = 2 w(ABC) = \pm 2$.
Этот пример показывает, что числа оборотов для разных ломаных с одинаковым объединением их отрезков могут быть разными.

\begin{example}\label{e:wn-surj} Для любых целого числа $n$ и точки $O$ существует замкнутая ломаная, число оборотов которой вокруг $O$ равно $n$.
\end{example}

\begin{proof}[Построение]
    Если $n=0$, то примером является замкнутая ломаная, состоящая из одной точки.
    Если $n\ne 0$, то возьмем правильный треугольник $ABC$ с центром $O$, ориентированный против часовой стрелки при $n>0$ (см. рисунок \ref{f:abco}) и по часовой стрелке иначе.
    Определим замкнутую ломаную $l$ формулой $l := \underbrace{ABC\ldots ABC}_{|n| \text{ раз}}$.
    Имеем $w(l) = |n| \cdot w(ABC) = n$.
\end{proof}

\begin{proposition}\label{p:noncl} 
    Число оборотов $w(A_1\ldots A_m)$ является целым числом. 
\end{proposition}

Доказательства утверждений этого параграфа см. в \cite{ABM+}.
Указание к другому доказательству утверждения \ref{p:noncl} см. в \cite[Hint after Proposition 3.4]{ANS}.

\begin{theorem}[Борсук--Улам] \label{t:borsuk}
    Пусть замкнутая ломаная $A_1\ldots A_{2k}$ не проходит через точку $O$ и симметрична относительно $O$ (т.~е. $O$ "--- середина отрезка $A_jA_{k+j}$ для каждого $j=1,\ldots,k$).
    Тогда число оборотов этой замкнутой ломаной вокруг точки $O$ нечетно.
\end{theorem}

Следующие обозначение и утверждение \ref{p:w'} будут полезны (например, для доказательства утверждений \ref{p:k23}, \ref{pr:off}, \ref{pr:triod}, построения примеров \ref{e:k23}, \ref{e:maps}, \ref{e:radonae} и введения определений в \S \ref{s:gawh}).

\textbf{Ломаной} называется упорядоченный набор точек (не обязательно различных).\footnote{В математике иногда разные вещи имеют одинаковые формализации, а разница проявляется в том, что мы с этими вещами делаем.
Тогда формально одинаковые объекты называются по-разному, чтобы сделать более понятными операции над ними.
Например, в этом тексте упорядоченный набор точек называется и ломаной, и замкнутой ломаной. 
}

Пусть $l=A_1\ldots A_m$ "--- ломаная, не проходящая через точку $O$. 
Определим действительное число $w'(l)=w'(l,O)$ формулой
$$2\pi \cdot w'(l):= \angle A_1OA_2+\angle A_2OA_3+\ldots+\angle A_{m-1}OA_m.$$
Очевидно, что

$\bullet$ $w'(A_1\ldots A_mA_1) = w(A_1\ldots A_m)$;

$\bullet$ $w'(A_1\ldots A_m) = w'(A_1\ldots A_j) + w'(A_j\ldots A_m)$ для каждого $j=1,\ldots,m$; 

$\bullet$ если точки $A_2,\ldots, A_{m-1}$ лежат внутри угла $\angle A_1 O A_m$, то  
$2\pi w'(A_1\ldots A_m) = \angle A_1 O A_m$.

\begin{proposition}\label{p:w'}
    Имеем $\angle A_1OA_m = 2\pi w'(A_1\ldots A_m)+2\pi n$ для некоторого целого $n$. 
\end{proposition}

\begin{proposition}\label{p:rel}
    Возьмем точки $P_0$ и $P_1$, соединенные ломаной, не пересекающейся с замкнутой ломаной $l$. 
    Тогда  $w(l,P_0)=w(l,P_1)$. 
\end{proposition}

Обозначим через $l^{-1}$ ломаную, полученную из ломаной $l$ прохождением в противоположном порядке.
Очевидно, что $w'(l) = -w'(l^{-1})$. 

\emph{Конкатенацией} ломаных $l_1 = A_1 \ldots A_m C$ и $l_2 = C B_1\ldots B_k$ называется  ломаная 
$$l_1l_2 := A_1 \ldots A_m C B_1 \ldots B_k.$$
Имеем $w'(l_1l_2) = w'(l_1) + w'(l_2)$ для любых точки $O$ и ломаных $l_1, l_2$, не проходящих через $O$.

\emph{Конкатенацией} замкнутых ломаных 
$l_1 = A_1 \ldots A_m C$ и $l_2 = B_1\ldots B_k C$ называется замкнутая ломаная 
$$l_1l_2 := A_1 \ldots A_m C B_1 \ldots B_k C.$$

\begin{proposition}\label{p:k23}
    Пусть $O$, $A$, $B$ "--- три попарно различные точки.
    Пусть $l_1$, $l_2$, $l_3$ "--- ломаные, соединяющие точку $A$ с точкой $B$, и не проходящие через точку $O$.
    Тогда 
    \linebreak
    $w(l_1 l_2^{-1}) + w(l_2 l_3^{-1}) = w(l_1 l_3^{-1})$.
\end{proposition}

\begin{proof} 
    Имеем
    $$w(l_1 l_2^{-1}) + w(l_2 l_3^{-1}) = w'(l_1) + w'(l_2^{-1}) + w'(l_2) + w'(l_3^{-1}) = w'(l_1) + w'(l_3^{-1}) = w(l_1 l_3^{-1}).$$  
\end{proof}

\begin{example}[ср. утверждение \ref{p:k23}; см. построение в \S\ref{s:wind}.A]\label{e:k23}
    Пусть $O$, $A$, $B$ "--- три попарно различных точки.
    Для любых трех целых чисел $n_1, n_2, n_3$, таких что $n_1 + n_2 = n_3$, существуют три ломаные $l_1, l_2, l_3$, соединяющие точку $A$ с точкой $B$, не проходящие через точку $O$, и такие, что
    $$w(l_1l_2^{-1}) = n_1, \qquad w(l_2l_3^{-1}) = n_2 \quad\text{и}\quad w(l_1l_3^{-1}) = n_3.$$
\end{example}

Простой пример \ref{e:k23} предваряет более интересные пример \ref{e:maps} и теорему \ref{t:k4}.

\section*{\ref{s:wind}.A. Построение примера \ref{e:k23}}

Для построения примеров в этом тексте не нужны \emph{сложные} картинки --- достаточно \emph{простых} преобразований \emph{простых} картинок. 
Например, полезно преобразование на рисунке \ref{f:fingermove}. 
На этом и других рисунках мы изображаем \emph{кривыми} ломаные с большим числом звеньев. Вершины ломаной расположены на кривой в порядке, обозначенном стрелками. 

\begin{figure}[ht] \centering
\compile{
    \includegraphics[scale=.9]{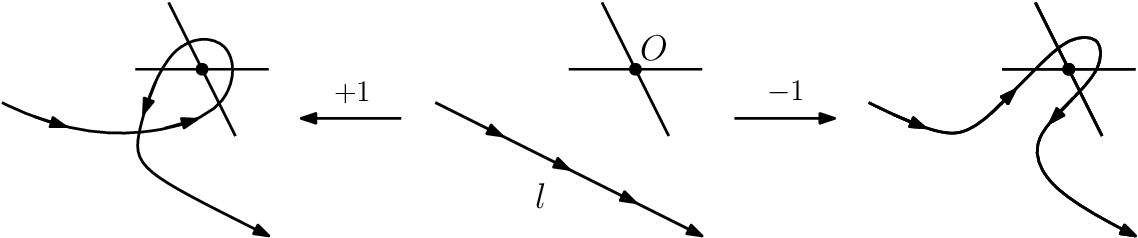}
    }
    \caption{<<Пальцевые движения>> ломаной $l$ вокруг точки $O$: положительное (слева) и отрицательное (справа)}
    \label{f:fingermove}
\end{figure}

\begin{proof}[Построение примера \ref{e:k23}]
    Возьмем такую точку $C$, что ломаная $ACB$ не проходит через точку $O$. 
    Положим $l_1 = l_2 = l_3 = ACB$.
    Сделаем
    
    $\bullet$ $|n_1|$ положительных/отрицательных <<пальцевых движений>> (рисунок~\ref{f:fingermove}) ломаной $l_1$ вокруг точки $O$, если $n_1$ положительно/отрицательно соответственно;

    $\bullet$ $|n_2|$ положительных/отрицательных <<пальцевых движений>> ломаной $l_3^{-1}$ вокруг точки $O$, если $n_2$ положительно/отрицательно соответственно.

    Получим $w(l_1 l_2^{-1}) = n_1$, $w(l_2 l_3^{-1}) = n_2$ и, ввиду утверждения \ref{p:k23}, $w(l_1l_3^{-1}) = n_3$.
\end{proof}

\begin{remark}\label{r:rigor} 
В вышеприведенном рассуждении <<пальцевые движения>> определены картинкой, а не строгим построением. 
Иногда это считается достаточным для такого рода результатов в научных журналах.  
Однако мы приведем более строгое построение примера \ref{e:k23}. 
Для более сложных результатов такая строгость уже является необходимой. 

Более строгое построение оказывается более коротким.  
Хотя оно основано на идее <<пальцевых движений>>, оказывается более экономным не упоминать в нем <<пальцевых движений>>. 
(Der Mohr hat seine Arbeit getan, der Mohr kann gehen. F. Schiller.\footnote{Мавр сделал свое дело, мавр может уйти. Ф. Шиллер.})
Так часто бывает при наведении строгости.  

\aronly{
Строгие доказательства с явным использованием пальцевых движений получаются применением леммы \ref{l:f-move}.
}
\end{remark}

Пусть $l$ --- замкнутая ломаная и $n>0$ целое.
Обозначим через

$\bullet$ $l^0$ замкнутую ломаную, состоящую из одной точки; 

$\bullet$ $l^n$ замкнутую ломаную $\underbrace{l \ldots l}_{n \text{ раз}}$;

$\bullet$ $l^{-n} := (l^{-1})^n$.

Очевидно, что для любых целого $n$ и точки $O\not\in l$ выполнено $w(l^n) = n \cdot w(l)$.
 
\begin{proof}[Более явное построение примера \ref{e:k23}] 
    Выберем точки $C, D, E$ таким образом, что замкнутая несамопересекающаяся ломаная $l := ACBDE$ ориентирована по часовой стрелке, и точка $O$ лежит внутри нее.
    Положим $l_2 := ACB, \quad l_1 := l^{n_1}l_2,\quad l_3 := l^{-n_2}l_2$.
    Тогда  
    $$w(l_1 l_2^{-1}) = w'(l_1) + w'(l_2^{-1}) = w'(l^{n_1}l_2) + w'(l_2^{-1}) = n_1 w(l) + w(l_2l_2^{-1}) = n_1,$$
    $$w(l_2 l_3^{-1}) = w'(l_2) + w'(l_3^{-1}) = w'(l_2) + w'(l_2^{-1}l_3^{n_2}) = w(l_2l_2^{-1}) + n_2 w(l) = n_2.$$
    Тогда $w(l_1l_3^{-1}) = n_3$ по утверждению \ref{p:k23}. 
\end{proof}

\section{Оборотное число: определение и обсуждение}\label{s:wingra}

\begin{remark}[некоторые строгие определения]

{\it Графом} (конечным) $(V,E)$ называется конечное множество $V$ вместе с набором $E\subset {V\choose 2}$ его двухэлементных  подмножеств (т.~е. неупорядоченных пар несовпадающих элементов).
(Общепринятый термин для этого понятия "--- {\it граф без петель и кратных ребер} или {\it простой граф}.)
Элементы данного конечного множества называются {\it вершинами}.
Пары вершин из $E$ называются {\it ребрами}.

Говоря нестрого, граф планарен, если его можно нарисовать <<без самопересечений>> на плоскости. 
Строго говоря, граф называется {\it планарным} (или кусочно-линейно вложимым в плоскость), если существует его \emph{вложение} в плоскость, т.~е.

$\bullet$ набор точек плоскости, соответствующих вершинам,  

$\bullet$ вместе с набором ломаных на плоскости, каждая из которых соединяет те пары из набора точек, которые соответствуют ребрам графа,
 
$\bullet$ причем ломаные несамопересекающиеся, и никакая из ломаных не пересекает внутренность другой ломаной.\footnote{Тогда любые две ломаные либо не пересекаются, либо пересекаются только по их общей концевой вершине.
Мы не требуем, чтобы <<ни одна изолированная вершина не лежала ни на одной из ломаных>>, поскольку этого свойства можно добиться малым шевелением.}
\end{remark}

Обозначим через

$\bullet$ $K$ произвольный конечный граф;

$\bullet$ $[n]$ множество $\{1, 2, \ldots, n\}$;

$\bullet$ $K_n$ полный граф на множестве  $[n]$ вершин;   

$\bullet$ $K_{m,n}$ полный двудольный граф с долями  $[m]$ и $[n]'$ (мы обозначаем через $A'$ копию множества~$A$).  

Мы рассматриваем изображения графов на плоскости, при которых ребра изображаются ломаными (и допускаются пересечения этих ломаных).
Строго говоря, {\bf отображением} (PL, кусочно-линейным) $f:K\to\R^2$ называется

$\bullet$ набор точек плоскости, соответствующих вершинам,  

$\bullet$ вместе с набором ломаных на плоскости, каждая из которых соединяет те пары из набора точек, которые соответствуют ребрам графа. 

\aronly{
(Это то же, что и первые две жирные точки из определения планарности.)  
}

Более точно, каждому ребру соответствует не ломаная, а пара ломаных, получаемых друг из друга прохождением в противоположном порядке. 
Задание ориентации на ребре графа задает выбор одной из этих двух ломаных. 

\begin{figure}[ht] \centering
\compile{
    \includegraphics[scale=0.7]{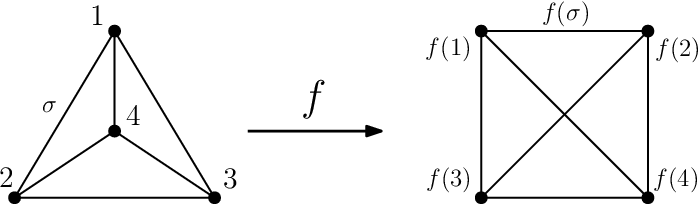}
    \quad 
    \includegraphics[scale=0.7]{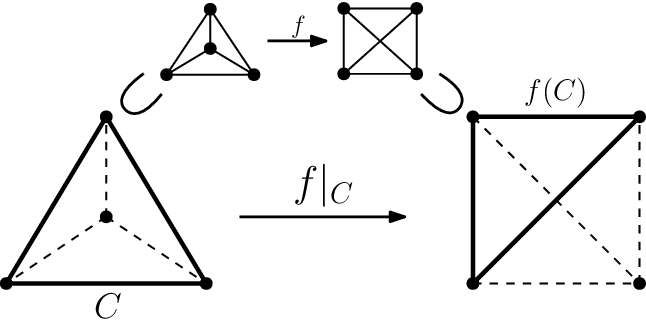}
    }
    \caption{Отображение $f: K_4 \to \R^2$ (слева); образ $f(C)$ и сужение $f|_C$ (справа)}
    \label{f:restr}
\end{figure}

{\bf Образ $f(\sigma)$} ребра $\sigma$ (при отображении $f$) "--- это объединение отрезков соответствующей ломаной (см. сноску \ref{fn:sets}). 
{\bf Образ} набора ребер "--- это объединение образов всех ребер из набора.

{\bf Сужение} $f|_{ab}$ на ориентированное ребро $ab$ "--- это соответствующая ломаная с началом $f(a)$ и концом $f(b)$.
Последовательность $v_1\ldots v_n$ попарно различных вершин графа $K$ называется (простым) \textbf{ориентированным циклом}, если $v_1v_2, \ldots, v_{n-1}v_n, v_nv_1$ --- ребра графа $K$.
Ориентация цикла $v_1\ldots v_n$ задает ориентацию на ребрах $v_1v_2, \ldots, v_{n-1}v_n, v_nv_1$ (хотя граф $K$ неориентированный).
{\bf Сужением} $f|_C:C\to\R^2$  на ориентированный цикл $C = v_1\ldots v_n$ в графе $K$ называется замкнутая ломаная, полученная конкатенацией ломаных $f|_{v_1v_2},\ldots,f|_{v_{n-1}v_n},f|_{v_nv_1}$ (рисунок \ref{f:restr}).

На этом языке утверждение \ref{p:k23} переформулируется так: для любого отображения 
$f: K_{3,2} \to \R^2$ и точки $O\not\in f(K_{3,2})$ имеем $w(f|_{1'12'2}) + w(f|_{1'22'3}) = w(f|_{1'12'3})$. 

Для $j \in [4]$ обозначим через $C_j$ ориентированный цикл в $K_4$, полученный удалением $j$ из последовательности $1234$.

\begin{proposition}\label{p:rel1}  
    Для любых отображения $f:K_4\to\R^2$ и точки $O\in\R^2-f(K_4)$ выполнено 
    $$\sum_{j=1}^4(-1)^jw(f|_{C_j})=0:\qquad -w(f|_{234})+w(f|_{134})-w(f|_{124})+w(f|_{123})=0.$$ 
\end{proposition}

Напомним, что $K-e$ --- граф, полученный из графа $K$ удалением ребра $e$.  

\begin{proof}[Доказательство утверждения \ref{p:rel1}]
Вот еще одна переформулировка утверждения \ref{p:k23}: 
\linebreak
$w(f|_{123}) + w(f|_{134}) = w(f|_{1234})$ для любого отображения $f: K_4 - 24 \to \R^2$ и точки $O\not\in f(K_4 - 24)$.  
Поэтому $w(f|_{123}) + w(f|_{134}) = w(f|_{1234}) = w(f|_{2341}) = w(f|_{234}) + w(f|_{124}).$
\end{proof}
  
\begin{example}[ср. утверждение \ref{p:rel1}]\label{p:mapso}
    Если $n_1-n_2+n_3-n_4=0$ для целых $n_1,n_2,n_3,n_4$, то существуют точка $O$ и отображение $f:K_4\to\R^2$, такое что $O\not\in f(K_4)$ и $w(f|_{C_j})=n_j$ для каждого $j\in[4]$. 
\end{example}

Пример \ref{p:mapso} строится аналогично примеру \ref{e:k23}.

\medskip
Для вершины $v$ и ориентированного цикла $C$ в графе $K$, таких что $f(v) \not\in f(C)$, назовем \textbf{оборотным числом} число
$$w_f(C,v):=w(f|_C,f(v)).$$
Назовем отображение $f : K \to \R^2$ \textbf{слабым почти вложением}, если $f(v)\not\in f(\sigma)$ для каждой пары из вершины $v$ и ребра $\sigma \not \ni v$. 
Этого свойства достаточно для того, чтобы все оборотные числа $w_f(C, v)$ для любых цикла $C$ и вершины $v \not \in C$ в графе $K$ были определены.
Любое почти вложение является слабым почти вложением.
Обратное верно только для некоторых графов, например, для $K_{3, 1}$ и $K_3$.

Для слабого почти вложения $f: K_4 \to \R^2$ и вершины $j$ графа $K_4$ обозначим 
$$w_f(j) := w_f(C_j, j).$$

\begin{proposition}\label{p:autom} 
Пусть $f:K_4\to\R^2$ --- слабое почти вложение, а $\sigma$ --- перестановка множества $[4]$ и соответствующий автоморфизм графа $K_4$. 
Тогда 
$$w_{f\circ\sigma}(j) = (-1)^{\sigma(j)-j}\sgn\sigma \cdot w_f(\sigma(j)) \quad \text{для каждого} \quad j = 1, 2, 3, 4.$$
\end{proposition}
\begin{proof}[Доказательство]
Разложим перестановку $\sigma$ в композицию транспозиций и воспользуемся следующим наблюдением: если $\sigma$ меняет местами 1 с 2 и оставляет на месте 3, 4, а $g:=f\circ\sigma$, то 
$$w_g(1) = w_f(2),\ \ w_g(2) = w_f(1),\ \ w_g(3) = -w_f(3),\ \ w_g(4) = -w_f(4).$$
\end{proof} 

\begin{figure}[ht] \centering
\compile{
    \includegraphics[scale=.6]{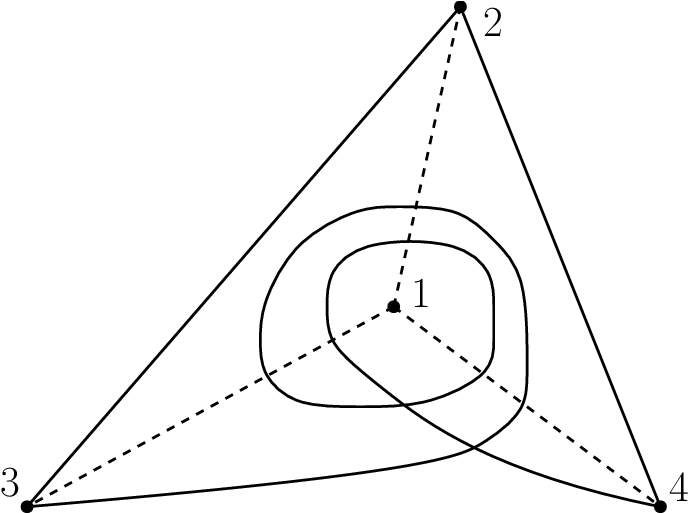}
    }
    \caption{Отображение $f:K_4\to\R^2$, такое что $w_f(234, 1)=3$}
    \label{f:k4-ex}
\end{figure}

Следующий пример показывает, что для слабых почти вложений $f: K_4 \to \R^2$ наборами оборотных чисел $w_f(j),~ j\in[4]$ реализуема любая четверка целых чисел.
 
\begin{example}[ср. рисунок~\ref{f:k4-ex}]\label{e:maps}
    Для любых целых $n_1,n_2,n_3,n_4$ существует слабое почти вложение $f:K_4\to\R^2$, такое что $w_f(j)=n_j$ для каждого $j \in [4]$.
\end{example}

Здесь и далее на рисунках вместо $f(i)$ мы коротко пишем $i$.

Простой пример \ref{e:maps} предваряет более интересную теорему \ref{t:k4}.

\aronly{
\section*{\ref{s:wingra}.A. Построение примера \ref{e:maps}} 
}

\begin{proof}[Построение примера \ref{e:maps} для $n_2 = n_3 = n_4 = 0$]
    Квадрат с диагоналями образует слабое почти вложение $g:K_4\to\R^2$; пронумеруем его вершины $g(j),~ j \in [4]$ против часовой стрелки.
    Тогда $w_g(j) = 0$ для каждого $j \in [4]$.
    
    Сделаем $|n_1|$ положительных/отрицательных <<пальцевых движений>> (рисунок~\ref{f:fingermove}) ломаной $g|_{23}$ вокруг точки $g(1)$, если $n_1$ положительно/отрицательно соответственно. 
    Причем сделаем эти преобразования так, чтобы полученное отображение $f$ было слабым почти вложением.
    Тогда $w_f(1)=n_1$ и $w_f(j)=0$ для каждого $j=2,3,4$.
\end{proof}

\aronly{
\begin{proof}[Построение примера \ref{e:maps}]
    Возьмем слабое почти вложение $f$ из идеи построения примера \ref{e:maps} для $n_2 = n_3 = n_4 = 0$.
    Сделаем
   
    $\bullet$ $|n_2|$ положительных/отрицательных <<пальцевых движений>> ломаной $f|_{34}$ вокруг точки $f(2)$, если $n_2$ положительно/отрицательно соответственно;
    
    $\bullet$ $|n_3|$ положительных/отрицательных <<пальцевых движений>> ломаной $f|_{41}$ вокруг точки $f(3)$, если $n_3$ положительно/отрицательно соответственно;
    
    $\bullet$ $|n_4|$ положительных/отрицательных <<пальцевых движений>> ломаной $f|_{12}$ вокруг точки $f(4)$, если $n_4$ положительно/отрицательно соответственно.
    
    Эти три преобразования можно сделать так, чтобы полученное в результате отображение было слабым почти вложением.
    Оно является искомым.
\end{proof}

Вышеприведенное построение можно коротко сформулировать так:
для каждого $j \in [4]$ и $j$-ого ориентированного ребра $\sigma_j$ из четверки $(23, 34, 41, 12)$ последовательно сделаем $|n_j|$ положительных/отрицательных <<пальцевых движений>>  ломаной $f|_{\sigma_j}$ вокруг точки $f(j)$, если $n_j$ положительно/отрицательно соответственно.

Мотивировка дальнейшего аналогична замечанию \ref{r:rigor}. 
}

\begin{proof}[Более явное построение примера \ref{e:maps}]
    В дальнейших равенствах цифры слева от знака равенства являются вершинами графа $K_4$, справа --- точками плоскости, изображенными на рисунке \ref{f:weak-emb}. 

    \begin{figure}[ht]\centering
    \compile{
    \includegraphics[scale=.9]{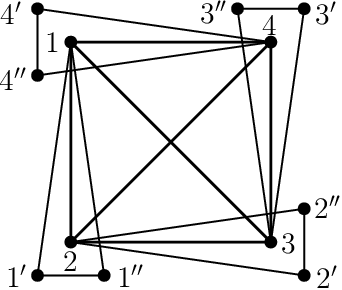}
    }
    \caption{К построению примера \ref{e:maps}}
    \label{f:weak-emb}
    \end{figure}

    Положим $f(j) = j$ для каждого $j \in [4]$,  $f|_{13} = 13$, $f|_{24} = 24$ и
    $$f|_{12} = 1(22'2'')^{n_3}2, \quad f|_{23} = 2(33'3'')^{n_4}3, \quad f|_{34} = 3(44'4'')^{n_1}4, \quad f|_{41} = 4(11'1'')^{n_2}1.$$
    Докажем, что отображение $f$ является искомым.
    Имеем 
    $$w_f(1) = w_f(234, 1) = w'(f|_{23}, f(1)) + w'(f|_{34}, f(1)) + w'(f|_{42}, f(1)) = $$
    $$ = w'(2(33'3'')^{n_4}3, 1) 
    + w'(3(44'4'')^{n_1}4, 1) + w'(42, 1) = n_4w(33'3'', 1) + n_1w(44'4'', 1) + w(234, 1) = n_1.$$
    Для остальных чисел оборотов доказательства аналогичны.
    (Похожие вычисления уже встречались в построении примера \ref{e:k23}.)
\end{proof}

\aronly{
Приведем другое построение примера \ref{e:maps}, полезное для обобщений.

Для вершины $v$ и ориентированного ребра $\tau$ в графе $K$, таких что $f(v) \not\in f(\tau)$, положим
$$w'_f(\tau,v):=w'(f|_\tau,f(v)).$$

\begin{lemma}[ср. {\cite[Lemma 2.1]{KS20}}]\label{l:f-move}
    Для любых графа $K$, его ориентированного ребра $ab$, его вершины $j \not \in ab$, знака $s \in \{-1, +1\}$ и слабого почти вложения $g : K \to \R^2$ существует слабое почти вложение $f : K \to \R^2$, такое что
    
    (1) $w'_f(ab, j) = w'_g(ab, j)  + s$  и 
    
    (2) $w'_f(\sigma, v) = w'_g(\sigma, v)$ для любых ориентированного ребра $\sigma\ne ab,ba$ и вершины $v\ne j$, $v \not \in \sigma$. 
\end{lemma}

\begin{proof}[Доказательство]
    Пусть имеется ломаная $l = A_1 \ldots A_m$ и треугольник с вершинами $X$, $Y$, $Z$, обозначенными против часовой стрелки.
    Точка $O$ лежит внутри треугольника $XYZ$.

   \begin{figure}[ht]\centering
   \compile{
        \includegraphics[scale=.9]{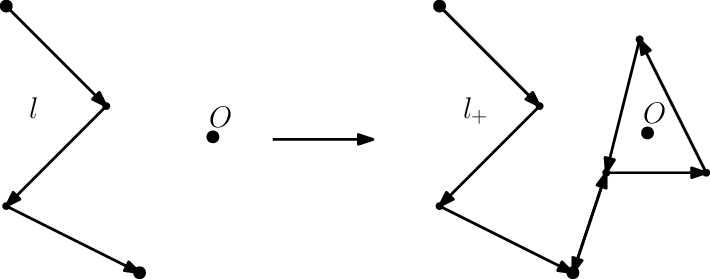}
        }
        \caption{Положительное пальцевое движение ломаной $l$ вокруг точки $O$}
        \label{f:fingermove-strict}
    \end{figure}
    
    Будем говорить, что ломаная $l_+$ получена из ломаной $l$ \emph{положительным пальцевым движением} вокруг точки $O$, если $l_+ = A_1 \ldots A_m XYZX A_m$.
    Аналогично $l_-$ получена из ломаной $l$ \emph{отрицательным пальцевым движением} вокруг точки $O$, если $l_- = A_1 \ldots A_m XZYX A_m$.
    
    Легко видеть, что $w'(l) = w'(l_+) - 1 = w'(l_-) + 1$.

    \smallskip
    Сделаем пальцевое движение ломаной $g|_{ab}$ вокруг $g(j)$ так, что 
    
    (i) $g(j)$ единственный образ вершины графа внутри треугольника $XYZ$,

    (ii) отрезок $g(b)X$ не содержит образов вершин графа $K$, отличных от $a$ и $b$,

    (iii) движение положительное, если $s = +1$, иначе отрицательное.
    
    Отображение $f$, полученное в результате этого пальцевого движения, является искомым.
    Условия (i,~iii) влекут условия (1,~2).
    Отображение $f$ является слабым почти вложением ввиду условия (ii).
\end{proof}

\begin{proof}[Другое построение примера \ref{e:maps}]
    Квадрат с диагоналями образует слабое почти вложение $g:K_4\to\R^2$; пусть вершины $g(j),~ j \in [4]$ пронумерованы против часовой стрелки.
    Тогда $w_g(C_j,j) = 0$ для каждого $j \in [4]$.

    Многократным применением леммы \ref{l:f-move} к отображению $g$ получим слабое почти вложение $f : K_4 \to \R^2$, такое что
    $$\begin{array}{l}
        w'_f(23, 1) = w'_g(23, 1) + n_1, \quad w'_f(34, 2) = w'_g(34, 2) + n_2,\\
        w'_f(41, 3) = w'_g(41, 3) + n_3, \quad w'_f(12, 4) = w'_g(12 ,4) + n_4
    \end{array}\quad \text{ и }$$
    $$w'_f(\sigma, v) = w'_g(\sigma, v) \text{ для каждой пары } (\sigma, v),~ v \not \in \sigma, \text{ отличной от перечисленных}.$$
    Ввиду равенства $w_f(ijk,v)=w'_f(ij,v)+w'_f(jk,v)+w'_f(ki,v)$ для любого отображения $f$, цикла $ijk$ и вершины $v \not \in ijk$, отображение $f$ является искомым.
\end{proof}
}


\section{Оборотные числа почти вложений графа $K_4$}
\label{s:alem-k4}

\begin{example}
\label{e:k3} 
    (a) Для любого целого $n$ существует почти вложение $f:K_3 \sqcup\{4\}\to\R^2$, такое что $w_f(4)=n$.
        
    (b) Для любого целого $n$ существует почти вложение $f:K_4\to\R^2$, такое что $w_f(4)=n$. (См.~рисунок~\ref{f:123k4} для $n=3$.)

    (с) Для любого вложения $f:K_3\sqcup\{4\}\to\R^2$ имеем $w_f(4)\in\{-1,0,1\}$ (это утверждение близко к теореме Жордана).
\end{example}

\begin{figure}[ht]\centering
\compile{
    \includegraphics[scale=.16]{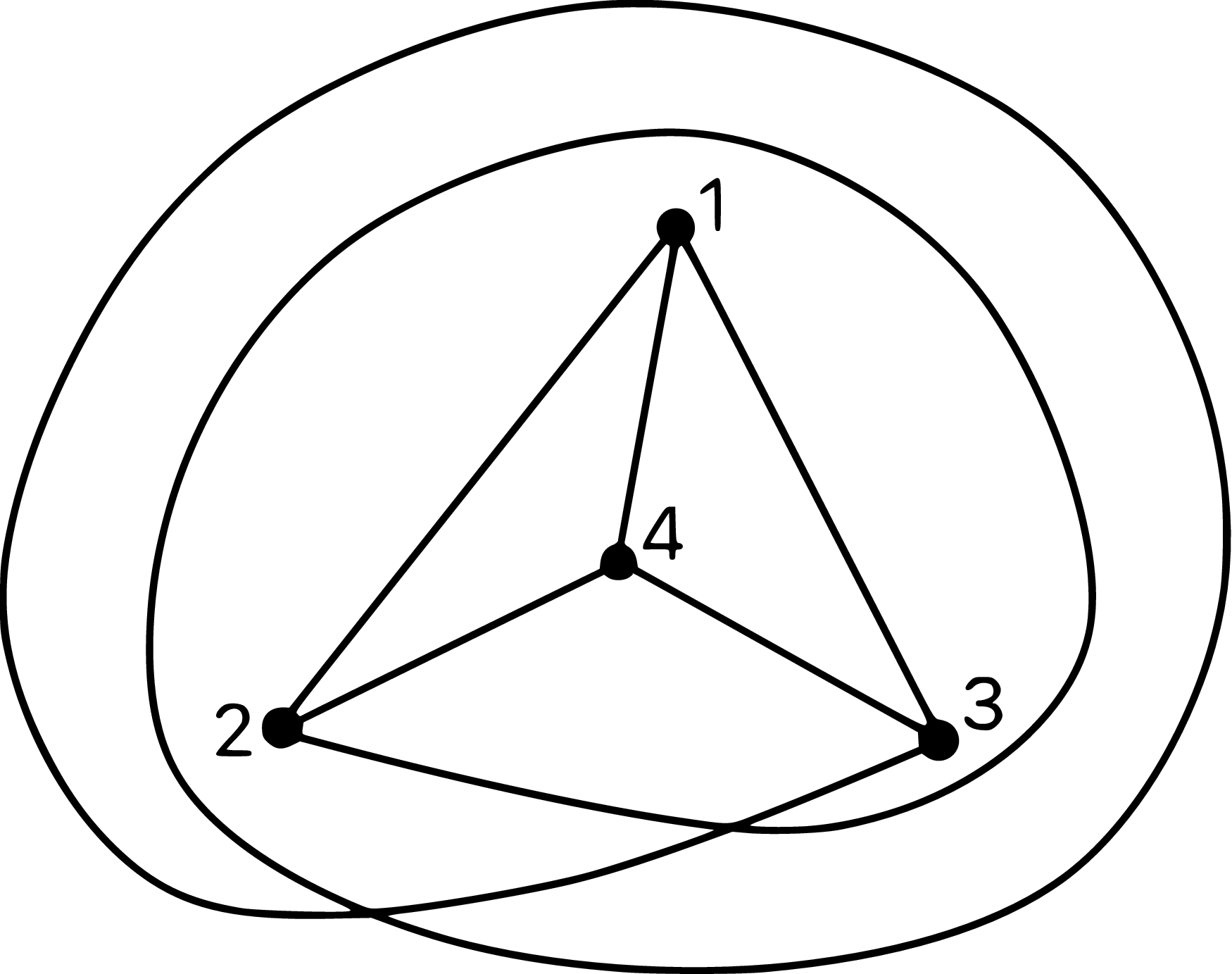}  
    }
    \caption{ Почти вложение $f:K_4\to\R^2$, такое что $w_f(4)=3$}\label{f:123k4}
\end{figure}

Для слабого почти вложения $f:K_4\to\R^2$ обозначим  
$$W_f:=\sum_{j=1}^4 (-1)^jw_f(j) = -w_f(234,1)+w_f(134,2)-w_f(124,3)+w_f(123,4).$$
Знак $(-1)^j$, не влияющий на четность, введен ради кососимметричности: по утверждению~\ref{p:autom} имеем $W_{f\circ\sigma} = \sgn\sigma \cdot W_f$ для любой перестановки $\sigma$ множества $[4]$.

\begin{theorem}\label{t:radonae} 
    Для любого почти вложения $f:K_4\to\R^2$ число $W_f$ нечетно. 
\end{theorem}

\begin{figure}[ht]\centering
    \includegraphics[scale=1]{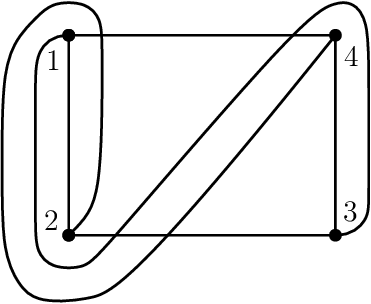}
\caption{(Е. Морозов) Почти вложение $f:K_4\to\R^2$, такое что $W_f=3\ne\pm 1$}
\label{f:k4-pm3}
\end{figure}

Аналог теоремы \ref{t:radonae} для вложений вместо почти вложений выглядит проще (он близок к теореме Жордана).
Более того, для любого вложения $f:K_4\to\R^2$ три из четырех чисел из теоремы \ref{t:radonae} равны нулю, а оставшееся равно $\pm1$. 
А если $f$ лишь почти вложение, то число $W_f$ может принимать значения, отличные от $\pm 1$ (см. рисунок \ref{f:k4-pm3}).
Аналог теоремы \ref{t:radonae} для слабых почти вложений вместо почти вложений неверен по утверждению \ref{e:maps}.
В отличие от утверждения \ref{p:rel1}, теорема \ref{t:radonae} не следует из <<соотношения $123+134+142+243=0$ в графе>>. 
Доказательства теоремы~\ref{t:radonae} приведены в \S\ref{s:alem}.A,~\S\ref{s:gawhap}.

\begin{theorem}[\cite{AM25}, ср. теорему \ref{t:radonae}]\label{t:k4} 
Для любых целых чисел  $n_1,n_2,n_3,n_4$, сумма которых нечетна, существует почти вложение  $f:K_4\to\R^2$, такое что $w_f(j)=n_j$ для каждого $j \in [4]$:
$$w_f(234, 1) = n_1,\quad w_f(134, 2) = n_2,\quad w_f(124, 3) = n_3,\quad w_f(123, 4) = n_4.$$
\end{theorem}

\section*{\ref{s:alem-k4}.A. К доказательству теоремы~\ref{t:k4}} 

Теорема \ref{t:k4} следует из двух лемм \cite[Lemma 1 и Lemma 2]{AM25}.
К доказательству \cite[Lemma 1]{AM25} подводит построение примера \ref{ex:Wf}.
Следующий пример является частным случаем теоремы \ref{t:k4}, но его построение проще и подводит к доказательству \cite[Lemma 2]{AM25}.

\begin{example}[ср. пример \ref{e:maps}]\label{e:radonae}
    Для любых целых $n_1,n_2,n_3,n_4$, таких что $\sum_{j=1}^4 (-1)^j n_j = \pm 1$  существует почти вложение $f:K_4\to\R^2$, такое что $w_f(j) = n_j$ для каждого $j \in [4]$.
\end{example}

\begin{figure}[ht]\centering
    \compile{
    \includegraphics[scale=.8]{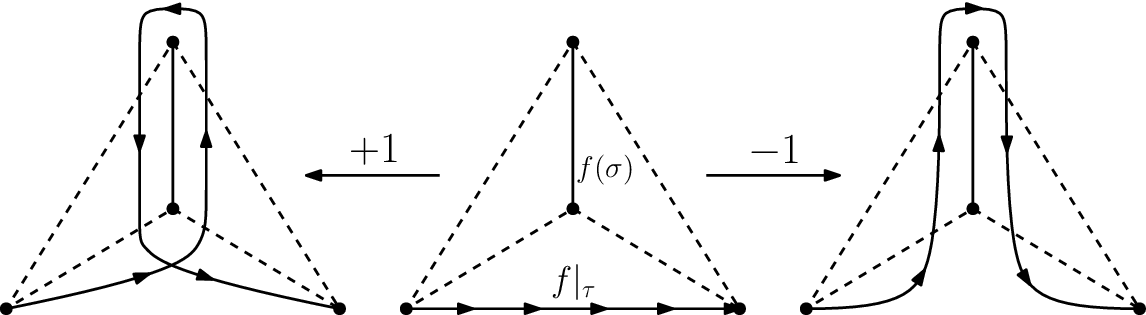}  
    }
    \caption{<<Пальцевые движения>> ломаной $f|_\tau$ вокруг отрезка $f(\sigma)$: положительное (слева), отрицательное (справа)}
    \label{f:finger-move} 
\end{figure}

\begin{proof}[Сведение примера \ref{e:radonae} к частному случаю]  
Достаточно построить почти вложение для $N := \sum_{j=1}^4 (-1)^j n_j = 1$.
Тогда нужное почти вложение существует и для $N=-1$, поскольку ввиду утверждения \ref{p:autom} имеем $W_{f\circ\sigma} = - W_f$ для любой нечетной перестановки $\sigma$ множества $[4]$. 
\end{proof}

\aronly{
\begin{proof}[Построение примера \ref{e:radonae} для частного случая $N = 1$.] 
    Правильный треугольник с его центром и рeбрами, соединяющими центр с вершинами, образуют отображение $f:K_4 \to \R^2$; пусть вершины $f(1),f(2),f(3)$ пронумерованы против часовой стрелки. Тогда $w_f(C_j, j) = 0$ для каждого $j \in [3]$, и $w_f(C_4, 4) = 1$.  
    Сделаем 

    $\bullet$ $|n_1|$ положительных/отрицательных <<пальцевых движений>> ломаной $f|_{23}$ вокруг отрезка $f(14)$ (рисунок~\ref{f:finger-move}, ср. рисунок~\ref{f:fingermove}), если $n_1$ положительно/отрицательно соответственно; 
    
    $\bullet$ $|n_2|$ положительных/отрицательных <<пальцевых движений>> ломаной $f|_{13}$ вокруг отрезка $f(24)$, если $n_2$ положительно/отрицательно соответственно; 
    
    $\bullet$ $|n_3|$ положительных/отрицательных <<пальцевых движений>> ломаной $f|_{12}$ вокруг отрезка $f(34)$, если $n_3$ положительно/отрицательно соответственно; полученное отображение является искомым.
\phantom\qedhere
\end{proof}
}

\begin{figure}[ht]\centering
\compile{
    \includegraphics[scale=.9]{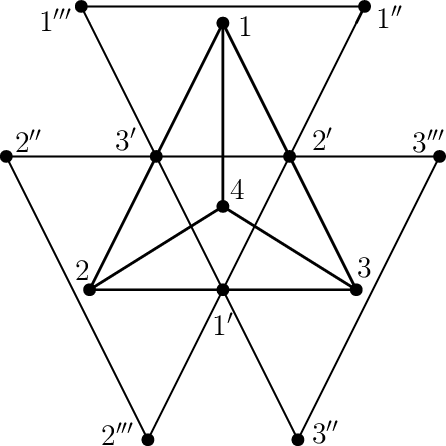}
    }
    \caption{К построению примера \ref{e:radonae}}
    \label{f:alm-emb}
\end{figure}

Следующее построение тоже основано на аналоге пальцевых движений для почти вложений (см. рисунок \ref{f:finger-move}; ср. рисунок \ref{f:fingermove}), но явно их не использует (см. замечание \ref{r:rigor}).  

\begin{proof}[Более явное построение примера \ref{e:radonae} для $N = 1$]
    В дальнейших равенствах цифры слева от знака равенства являются вершинами графа $K_4$, справа --- точками плоскости, изображенными на рисунке \ref{f:alm-emb}.
    Положим $f(j) = j$ для каждого $j \in [4]$,  $f|_{j4} = j4$ для каждого $j \in [3]$ и
    $$f|_{12} = 1(3'3''3''')^{n_3}3'2, \quad f|_{23} = 2(1'1''1''')^{n_1}1'3, \quad f|_{13} = 1(2'2''2''')^{n_2}2'3.$$
    Докажем, что отображение $f$ является искомым.
    Имеем 
    $$w_f(1) = w_f(234, 1) = w'(f|_{23}, f(1)) + w'(f|_{34}, f(1)) + w'(f|_{42}, f(1)) = $$
    $$ = w'(2(1'1''1''')^{n_1}1'3, 1) + w'(34, 1) + w'(42, 1) = 
    n_1w(1'1''1''', 1) + w(21'34, 1) = n_1.$$
    Для остальных чисел оборотов доказательства аналогичны.
    (Похожие вычисления уже встречались в построениях примеров \ref{e:k23} и \ref{e:maps}.)
\end{proof}

\section{Оборотные числа почти вложений графов} 
\label{s:alem}

Для почти вложений в плоскость тех графов, которые содержат подграф, изоморфный $K_4$, также справедливы соотношения, получающиеся из теоремы \ref{t:radonae}.
Например, в графе $K_5 - 45$ ровно два таких подграфа.


Для почти вложения $f : K_5 - 45 \to \R^2$ обозначим $W_f := W_{g}$, где $g$ "--- сужение отображения $f$ на $K_4$-подграф, содержащий вершину $4$.

\begin{proposition}
\label{p:k4-k5}
    Для любого почти вложения $f:K_5-45\to\R^2$ верно 
    $$W_{f} =  w_f(123,4)-w_f(123,5).$$ 
\end{proposition}


\begin{proof}
Утверждение
следует из следующей цепочки равенств 
$$w_f(123,5) \stackrel{(1)}{=} w_f(234,5) - w_f(134,5) + w_f(124,5) \stackrel{(2)}{=} w_f(234,1) - w_f(134,2) + w_f(124,3),$$
где равенство (1) верно по утверждению \ref{p:k23}.
В следующем абзаце докажем равенство (2).

Поскольку $f$ --- почти вложение, то $f(51)$ не пересекает $f(234)$.
Следовательно, $w_f(234, 5) = w_f(234, 1)$ ввиду утверждения \ref{p:rel}.  
Аналогично, $w_f(134, 5) = w_f(134, 2)$ и  $w_f(124, 5) = w_f(124, 3)$.
\end{proof}

\begin{remark} 
\label{r:k5-e}
Рассмотрим произвольное почти вложение $f:K_5-45\to\R^2$.

Соотношение из утверждения~\ref{p:k4-k5} можно переписать в следующем виде
$$
w_f(123,5) = w_f(234,1) - w_f(134,2) + w_f(124,3).
$$
Аналогично 
$$
w_f(123,4) = w_f(235,1) - w_f(135,2) + w_f(125,3).
$$
Это следует из утверждения~\ref{p:k4-k5} для почти вложения $f\circ\sigma$, где $\sigma$ "--- автоморфизм графа $K_5-45$, переставляющий вершины 4 и 5 местами. 
\end{remark}

\begin{theorem}[ср. теоремы \ref{t:radonae} и \ref{t:k4}]\label{t:k5-e} 
    Для любого почти вложения $f:K_5-45\to\R^2$ разность $w_f(123,4)-w_f(123,5)$ 
    
    (a) нечетна; 
    
    (b) равна $\pm1$. 
\end{theorem}

Теорема \ref{t:k5-e}.a следует из утверждения~\ref{p:k4-k5} и теоремы~\ref{t:radonae}.
Прямое доказательство теоремы~\ref{t:k5-e}.a приведено в \S\ref{s:alem}.A.
Теорема~\ref{t:k5-e}.b доказана в \cite{Ga23}. 

Аналог теоремы~\ref{t:k5-e}.b для вложений вместо почти вложений выглядит более просто (он близок к теореме Жордана).

\begin{figure}[ht]\centering
\compile{
\includegraphics[scale=0.5]{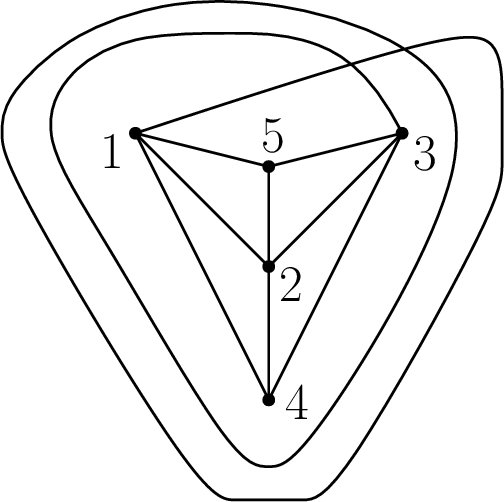} 
}
\caption{Почти вложение $f:K_5-45\to\R^2$, такое что $w_f(123, 5) = 3$} 
\label{f:k5e}
\end{figure}

\begin{example} 
\label{e:k5-map}
    (a) Аналог теоремы \ref{t:k5-e}.a для слабых почти вложений вместо почти вложений неверен.
    Контрпримером является отображение, переводящее вершины графа в вершины правильного пятиугольника, а его ребра --- в отрезки.
    Для этого примера \linebreak $w_f(123, 4)-w_f(123, 5) = 0$.
    
    (b) Для любого целого $n$ существует почти вложение $f:K_5-45\to\R^2$, такое что $w_f(123,5)=n$. 
    (Для $n=2$ см. рисунок \ref{k5} с той же нумерацией, что на рисунке \ref{f:k5e}. 
    Для $n=3$ см. рисунок \ref{f:k5e}.)
\end{example}


\begin{figure}[ht]\centering
     \includegraphics[width=4cm]{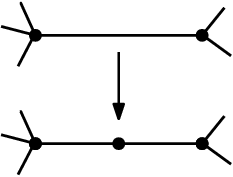}
     \caption{Подразделение ребра}\label{f:subdivision}
\end{figure}

Граф $K_{3,3} - 33'$ изоморфен графу, полученному из графа $K_4$ подразделением двух несмежных ребер (см. рисунок \ref{f:subdivision} и ср. рисунки \ref{f:k4-pm3} и \ref{f:k33-conj}).
Отождествим вершины $1, 2, 3, 4$ графа, полученного подразделением двух ребер $13$ и $24$ графа $K_4$, с вершинами $1, 1', 2, 2'$ графа $K_{3,3} - 33'$.
Назовем отображение $f : K_{3,3}-33' \to \R^2$ \textit{подразделением} отображения $g : K_4 \to \R^2$, если для $g|_{24} = A_1\ldots A_n$ и $g|_{13} = B_1 \ldots B_m$ выполнено
$$f(3) \in A_{i}A_{i+1} \text{ и } f(3') \in B_{j}B_{j+1} \text{ для некоторых } i \in [n-1] \text{ и } j \in [m-1];$$
$$f|_{1'3} = A_1 \ldots A_if(3), 
\ f|_{32'} = f(3)A_{i+1}\ldots A_n, 
\ f|_{13'} = B_1 \ldots B_jf(3'),  
\ f|_{3'2} = f(3')B_{j+1} \ldots B_m$$
и $f$ совпадает с $g$ на остальных вершинах и ребрах.

Нетрудно проверить, что если почти вложение $f : K_{3,3}-33' \to \R^2$ является подразделением отображения $g : K_4 \to \R^2$, то $g$ "--- почти вложение.






\begin{proposition}[ср. утверждение \ref{p:k4-k5}]
\label{p:k4-k33}
    Для почти вложения $f: K_{3,3} - 33' \to \R^2$ возьмем  почти вложение $g : K_4 \to \R^2$, подразделением которого $f$ является.
    Тогда $w_f(11'22', 3) - w_f(11'22', 3') = W_g$.
\end{proposition}

\begin{proof}[Доказательство (аналогично доказательству утверждения~\ref{p:k4-k5})]
	Напомним про отождествление вершин. Поэтому
    $$W_g = - w_g(1'22', 1) + w_g(122', 1') - w_g(11'2', 2) + w_g(11'2, 2') .$$
    Достаточно доказать, что 
    $$w_f(11'22', 3) = w_g(122', 1') + w_g(11'2, 2') \quad \text{и} \quad w_f(11'22', 3') = w_g(1'22', 1) + w_g(11'2', 2).$$
    Докажем первое из этих равенств, второе доказывается аналогично.
    Имеем
    $$w_f(11'22', 3) \stackrel{(1)}{=} w_f(11'23', 3) + w_f(13'22', 3) \stackrel{(2)}{=} w(g|_{11'2}, f(3)) + w(g|_{122'}, f(3)) \stackrel{(3)}{=} $$
    $$\stackrel{(3)}{=} w_g(11'2, 2') + w_g(122', 1'), \text{ где}$$
    
    $\bullet$ равенство $(1)$ верно по утверждению \ref{p:k23}; 
    
    $\bullet$ равенство (2) --- поскольку $f|_{11'23'} = g|_{11'2}$ и $f|_{13'22'} = g|_{122'}$.

    Докажем равенство (3). 
    Поскольку $f$ --- почти вложение, то $f(32')$ не пересекает $g(11'2)$.
    Следовательно, $w(g|_{11'2}, f(3)) = w_g(11'2, 2')$ ввиду утверждения \ref{p:rel}.\\
    Аналогично, $w(g|_{122'}, f(3)) = w_g(122', 1')$.
\end{proof}

\begin{theorem}[ср. теоремы \ref{t:radonae} и \ref{t:k5-e}.a]
\label{t:k33-odd} 
    Для любого почти вложения $f:K_{3, 3}-33'\to\R^2$ разность $w_f(11'22',3)-w_f(11'22',3')$ нечетна.
\end{theorem}

Теорема \ref{t:k33-odd} следует из утверждения~\ref{p:k4-k33} и теоремы~\ref{t:radonae}.
Прямое доказательство теоремы~\ref{t:k33-odd} аналогично доказательству теоремы \ref{t:k5-e}.a из \S\ref{s:alem}.A.

Аналог теоремы \ref{t:k33-odd} для вложений вместо почти вложений заключается в том, что $w_f(11'22',3)-w_f(11'22',3') = \pm 1$.
Это утверждение верно и близко к теореме Жордана.

\begin{figure}[ht]\centering
\includegraphics[scale=1]{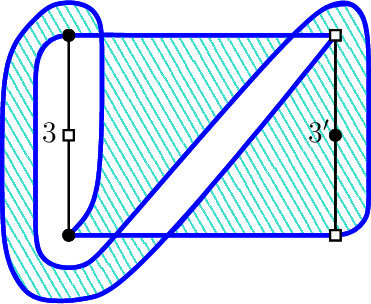}   
\caption{Почти вложение $f: K_{3,3} - 33' \to \R^2$, такое что \\ $w_f(11'22', 3) - w_f(11'22', 3') = \pm 3 \neq \pm1$; внутренность по модулю 2 ломаной $f|_{11'22'}$ заштрихована} 
\label{f:k33-conj}
\end{figure}

\begin{example}[ср. теоремы~\ref{t:k4} и~\ref{t:k5-e}.b]
\label{e:k33-strong}
Для любых целых чисел $n_1, n_2$ разной четности существует почти вложение $f:K_{3, 3}-33'\to\R^2$, такое что $w_f(11'22',3) = n_1$ и $w_f(11'22',3') = n_2$.
\end{example}

Построение примера~\ref{e:k33-strong} основано на утверждении~\ref{p:k4-k33} и идеях доказательства теоремы~\ref{t:k4}, и приведено в~\S\ref{s:alem}.B.

\begin{remark}
\label{r:int_analogs}
Аналоги теорем \ref{t:radonae}, \ref{t:k5-e}.a, \ref{t:k33-odd} верны для $\Z_2$-вложений (см. замечание \ref{r:zz2}). 
Гипотеза: аналог теоремы \ref{t:k5-e}.b верен для $\Z$-вложений, но неверен для $\Z_2$-вложений.
\end{remark}

\section*{\ref{s:alem}.A. 
Доказательства теорем \ref{t:radonae} и \ref{t:k5-e}.a} 

Несколько точек находятся в \emph{общем положении}, если никакие три из них не лежат на прямой и никакие три отрезка, их соединяющие, не имеют общей внутренней точки.

\begin{proposition}\label{p:stokes} 
    Возьмем замкнутую и незамкнутую ломаные $l$ и $p$, чьи вершины попарно различны и находятся в общем положении.
    Обозначим через $P_0$ и $P_1$ начальную и конечную точки ломаной $p$.
    Тогда $w(l,P_1)-w(l,P_0) \equiv |l\cap p| \mod2$.
    
\end{proposition}

Это классическое утверждение является дискретной версией \emph{теоремы Стокса}.
Его доказательство см. в \cite[Утверждение 2.2.a]{ABM+}.
 
Пусть $v_1, \ldots, v_n$ --- вершины графа $K$, а $e_1, \ldots, e_m$ --- его ребра.
Отображение $f:K\to\R^2$ называется отображением \emph{общего положения}, если
точки упорядоченного набора $\left(f(v_1), \ldots, f(v_n), A^{e_1}_1, \ldots, A_{l_{e_1}}^{e_1}, \ldots, A_1^{e_m}, \ldots, A_{l_{e_m}}^{e_m}\right)$ находятся в общем положении, где $A_1^e, \ldots, A_{l_e}^{e}$ --- все неконцевые вершины сужения $f|_e$ на ребро $e$.

Пусть $f:K\to\R^2$ --- отображение общего положения.
Тогда образы любых двух несмежных ребер пересекаются в конечном числе точек.
Обозначим через $V(f)$ количество точек пересечения образов несмежных ребер.

\begin{proof}[Сведение теоремы \ref{t:k5-e}.a к теореме \ref{t:hatuvkfl}] 
    Можно считать, что $f$ является отображением общего положения.
    Иначе заменим его <<малым шевелением>> на почти вложение $g: K_5-45 \to \R^2$ общего положения, такое что $w_g(123, 4) = w_f(123, 4)$ и $w_g(123, 5) = w_f(123, 5)$.
    
    Рассмотрим отображение $h: K_5 \to \R^2$ общего положения, такое что $h|_{K_5-45} = f$. Теорема \ref{t:k5-e}.a вытекает из следующей цепочки равенств и сравнений по модулю 2: 
    $$w_f(123, 4) - w_f(123, 5) = w_h(123, 4) - w_h(123, 5) \stackrel{(1)}{\equiv} |h(45) \cap h(123)| \stackrel{(2)}{=} V(h) \stackrel{(3)}{\equiv} 1.$$
    Здесь сравнение (1) верно ввиду утверждения \ref{p:stokes}, равенство (2) --- поскольку $f$ является почти вложением, а сравнение (3) является частным случаем следующей теоремы.
\end{proof}

\begin{theorem}[Ханани--Татт; ван Кампен--Флорес]
\label{t:hatuvkfl}
    Пусть граф $K$ --- это $K_5$ или $K_{3,3}$.
    Для любого отображения $h:K\to\R^2$ общего положения число $V(h)$ нечетно.
\end{theorem}

См. доказательство в обзоре {\cite[Lemma 1.4.3, Remark 1.4.4.a]{Sk18}}.

Теорема~\ref{t:k5-e}.b является \emph{целочисленной версией для почти вложений} случая $K = K_5$ теоремы \ref{t:hatuvkfl}. 
Заметим, что теорема \ref{t:hatuvkfl} не имеет \emph{целочисленной версии для отображений} (это известно и объяснено в \cite[Remark 5.b]{Ga23}).

\textit{Внутренностью по модулю 2} замкнутой ломаной $l$, чьи вершины попарно различны и находятся в общем положении, называется множество всех точек $X\in\R^2-l$, для которых найдется ломаная $p$, 

$\bullet$ соединяющая точку $X$ с точкой, расположенной достаточно далеко от $l$ (вне выпуклой оболочки вершин ломаной $l$), 

$\bullet$ пересекающая $l$ в нечетном числе точек, и 

$\bullet$ такая что все вершины ломаных $l$ и $p$ попарно различны и (кроме $X$) находятся в общем положении. 

Корректность этого определения следует из классической леммы о четности: \emph{если вершины двух замкнутых ломаных попарно различны и находятся в общем положении, то ломаные пересекаются в четном числе точек}. (См. доказательство в \cite[лемма 1.3.3]{Sk18}; эта лемма также есть частный случай утверждения \ref{p:stokes} для $P_0 = P_1$.)

\begin{proposition}\label{p:radon}
Пусть вершины ломаной $l$ попарно различны и находятся в общем положении.
Точка $O$ лежит во внутренности по модулю 2 замкнутой ломаной $l$ тогда и только тогда, когда $w(l)$ нечетно.
\end{proposition}

Для отображения $f : K_4 \to \R^2$ общего положения обозначим через $I(f)$ количество тех вершин $j$, образы которых лежат во внутренности по модулю 2 замкнутой ломаной $f|_{C_j}$.

\begin{theorem}[топологическая теорема Радона для плоскости]
\label{t:radon}
Для любого отображения $h:K_4\to\R^2$ общего положения сумма $V(h)+I(h)$ нечетна.
\end{theorem}

Эту теорему в эквивалентном виде доказали Эрвин~Г.~Баймочи и Имре~Барань в \cite{BB}.
См. доказательство в обзоре {\cite[Lemma 2.2.3]{Sk18}}.

 \begin{proof}[Сведение теоремы \ref{t:radonae} к теореме \ref{t:radon}] 
Можно считать, что $f$ является отображением общего положения.
Иначе заменим его <<малым шевелением>> на почти вложение $g: K_4 \to \R^2$ общего положения, такое что $w_g(j) = w_f(j)$ для каждого $j \in [4]$.

Теперь ввиду утверждения \ref{p:radon} имеем $W_f \equiv I(f) \pmod 2$.
Так как $f$ является почти вложением, то $V(f)=0$.
Тогда теорема~\ref{t:radonae} вытекает из  теоремы \ref{t:radon}.
\end{proof}

\section*{\ref{s:alem}.B. 
Построение примера \ref{e:k33-strong}} 

Следующий пример \ref{e:k33-weak} является частным случаем примера \ref{e:k33-strong}, но его построение проще и подводит к построению общего примера.
Пример \ref{e:k33-weak} первым построил А. Лазарев (не опубликовано).

\begin{example}
\label{e:k33-weak}
Для любого нечетного $n$ существует почти вложение $f:K_{3, 3}-33'\to\R^2$, такое что $w_f(11'22',3)-w_f(11'22',3')= n$. 
(См. рисунок \ref{f:k33-conj} для $n = \pm3$.)
\end{example}

Пример~\ref{e:k33-weak} является следствием утверждения~\ref{p:k4-k33} и следующего примера.

\begin{example}\label{ex:Wf}
	Для любого нечетного $n$ существуют почти вложения $g: K_4 \to \R^2$ и $f: K_{3,3} - 33' \to \R^2$, такие что $W_g = n$ и $f$ является подразделением отображения $g$.
\end{example}

\begin{figure}[ht]\centering
    \includegraphics[scale=0.6]{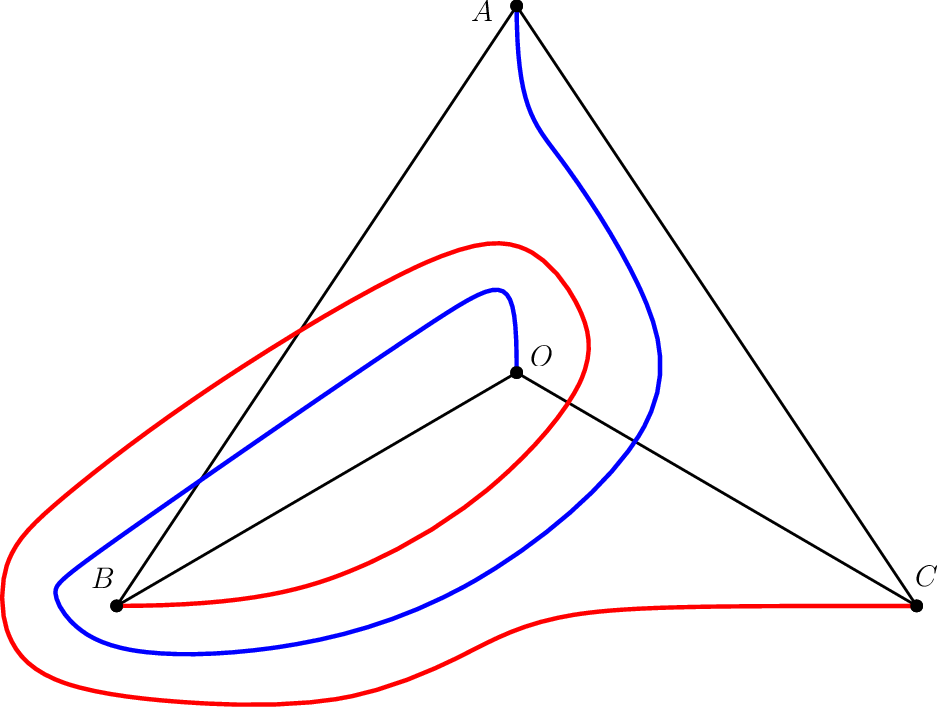}  
    \caption{К построению примера \ref{ex:Wf} (ср. рисунок \ref{f:k4-pm3})}\label{f:k4-k33}
\end{figure}

\begin{proof}[Построение примера \ref{ex:Wf}]


Возьмем правильный треугольник $ABC$ с центром $O$, вершины обозначены против часовой стрелки.
Построим почти вложение $g: K_4\to \R^2$ так (см. рисунок \ref{f:k4-k33}):

$\bullet$ $g$ отображает вершины $1,2,3,4$ графа $K_4$ в точки $A, B, C, O$ соответственно;

$\bullet$ сужения $A\ldots O$ и $B\ldots C$ взяты из \cite[Proof of Lemma 3]{AM25} с параметрами $\varepsilon_2 < |BO|/2$ и $m = (1 - n)/2$;

$\bullet$ сужения на остальные ребра --- отрезки.

В качестве $f$ возьмем почти вложение, полученное из $g$ подразделением отрезков $BO$ и $AC$ в серединах.
\footnote{Построение примера \ref{e:k33-weak} можно завершить уже здесь без применения утверждения~\ref{p:k4-k33}.
Вместо следующей проверки можно аналогично проверить, что разность чисел оборотов замкнутой ломаной $g|_{1234}$ вокруг середин отрезков $BO$ и $AC$ равна $n$.}
В следующем абзаце проверим, что $W_g = n$. 

По \cite[Proof of Lemma 3]{AM25} имеем
$$w_g(1) = w'(B\ldots C, A) + w'(CO, A) + w'(OB, A) = \frac{1}{6} - \frac{1}{12} - \frac{1}{12} = 0;$$
$$w_g(2) = w'(AC, B) + w'(CO, B) + w'(O \ldots A, B) = -\frac{1}{6} + \frac{1}{12} + \bigg(\frac{1}{12} - m\bigg) = \frac{n - 1}2;$$
$$w_g(3) = w'(AB, C) + w'(BO, C) + w'(O\ldots A, C) = \frac{1}{6} - \frac{1}{12} - \frac{1}{12} = 0;$$
$$w_g(4) = w'(AB, O) + w'(B\ldots C, O) + w'(CA, O) = \frac{1}{3} + \bigg(\frac{1}{3} - m\bigg) + \frac{1}{3} = 1 + \frac{n-1}2.$$
Откуда 
$$W_g = -w_g(1) + w_g(2) - w_g(3) + w_g(4) = -0 + \frac{n-1}2 - 0 + \bigg(1 + \frac{n-1}2\bigg) = n.$$ 
\end{proof}

\begin{proof}[Построение примера~\ref{e:k33-strong}]
Возьмем почти вложение $h : K_{3,3} - 33' \to \R^2$ из построения примера \ref{ex:Wf} для $n = n_1 - n_2$.
Тогда $w_h(11'22', 3) = n_1 - n_2$ и $w_h(11'22', 3') = 0$.

Модифицируем $h$, сделав $n_2$ положительных <<пальцевых движений>> отрезка $AB$ вокруг ломаной $UOCV$, где $U := h(3)$ и $V := h(3')$ --- середины отрезков $BO$ и $AC$ соответственно (см. рисунок \ref{f:finger-move}).
Получим новое почти вложение $f:K_{3,3} - 33' \to \R^2$, для которого $w_f(11'22', 3) = n_1$ и $w_f(11'22', 3') = n_2$.
\end{proof}

\section{Проблема описания значений оборотных чисел}\label{s:problem}

Говоря нестрого, для почти вложения $f:K\to\R^2$ рассмотрим набор оборотных чисел $w_f(C,v)$ для всевозможных пар $(C, v)$, образованных вершиной $v$ графа $K$ и ориентированным простым циклом $C$ в $K-v$.
Как описать наборы, реализуемые почти вложениями $f:K\to\R^2$?

Строго говоря, рассмотрим множество $H_K$ всевозможных пар $(C, v)$, образованных вершиной $v$ графа $K$ и простым циклом $C$ в $K-v$.
Для каждой пары $(C, v) \in H_K$ выберем ориентацию на цикле $C$ и далее считаем циклы $C$ ориентированными.
(Если цикл $C'$ получается из $C$ обращением ориентации, то $w_f(C', v) = -w_f(C, v)$.)
Почти вложение $f : K \to \R^2$ задает отображение $w_f :H_K \to \Z$ формулой $(C, v) \mapsto w_f(C, v)$.


\begin{problem}\label{pm:wind}
Для данного плоского графа $K$ опишите отображения $H_K \to \Z$, равные $w_f$ для некоторого почти вложения $f:K\to\R^2$.



Возможно, это описание не имеет красивой формулировки. 
Тогда интересен алгоритм, который по отображению $H_K \to \Z$ выясняет, равно ли оно $w_f$ для некоторого почти вложения $f:K\to\R^2$.
\end{problem}

Для аналогичной проблемы с заменой почти вложений на \emph{вложения} имеется неявное описание и, скорее всего, не имеется явного (см. объяснение в замечании~\ref{r:embeddings}).

См. аналогичную проблему~\ref{pm:Wu} с добавлением других инвариантов.

Решение проблемы \ref{pm:wind} для $K = K_4$ дается теоремами \ref{t:radonae} и \ref{t:k4}. 
Если $K$ --- лес или цикл, то проблема тривиальна, поскольку для этих графов $H_K = \varnothing$.
Если $K$ --- собственный подграф графа $K_4$, то никаких ограничений на реализуемые отображения $H_K \to \Z$ нет, см. пример \ref{e:k3}.a и следующий пример.

\begin{example}\label{p:k4-34}
    (a) Для любого целого $n$ существует почти вложение $f:K_4 - 24 - 34 \to \R^2$, такое что $w_f(4) = n$.

    (b) Для любых целых $n, m$ существует почти вложение $f: K_4 - 34 \to \R^2$, такое что $w_f(3) = n$ и $w_f(4) = m$.
    (См. рисунок \ref{f:k4-34} для $n = 2$ и $m = 2$.)
\end{example}

\begin{figure}[ht]\centering
    \compile{
    \includegraphics[scale=0.75]{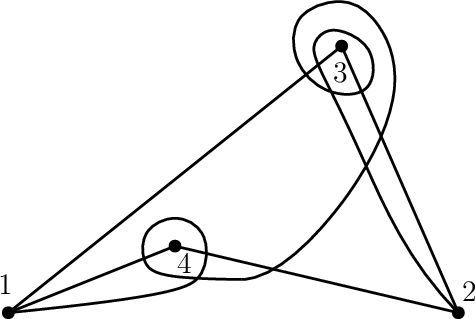}  
    }
    \caption{Почти вложение $f:K_4-34\to\R^2$, такое что $w_f(3)=2$ и $w_f(4)=2$}\label{f:k4-34}
\end{figure}

Проблема \ref{pm:wind} открыта даже когда $K$ есть $K_5-45$ (см. следующую гипотезу), граф куба, граф октаэдра.

\begin{conjecture}
\label{c:k5-e}
    Для любых целых чисел  $n_{41},n_{42},n_{43}$ и $n_{51},n_{52},n_{53}$, таких что
    разность     \linebreak
    $(n_{41} - n_{42} + n_{43}) - (n_{51} - n_{52} + n_{53})$ равна $\pm1$ (ср. замечание~\ref{r:k5-e}, теорему~\ref{t:k5-e}.b),
    существует почти вложение  $f:K_5-45\to\R^2$, такое что 
    $$w_f(234, 1) = n_{41},\quad w_f(134, 2) = n_{42},\quad w_f(124, 3) = n_{43}
    ,$$
    $$w_f(235, 1) = n_{51},\quad w_f(135, 2) = n_{52},\quad w_f(125, 3) = n_{53}
    .$$
\end{conjecture}

\section{Циклическое и триодическое числа: определения и обсуждение}\label{s:gawh}

Определим другие инварианты, и интересные сами по себе, и полезные (см.~\S\ref{s:gawhap}) для изучения оборотных чисел.

Назовем тройку ломаных $l_1,l_2,l_3$, соединяющих $A_1$ с $A_2$, $A_2$ с $A_3$, $A_3$ с $A_1$ соответственно (и, таким образом, образующих замкнутую ломаную), \textbf{циклической}, если $A_i\not\in l_{i+1}$ для каждого $i=1,2,3$ (нумерация по модулю 3; другими словами, если ломаные образуют почти вложение графа $K_3$ в плоскость).

\begin{figure}[ht] \centering
\compile{
    \includegraphics[scale=0.22]{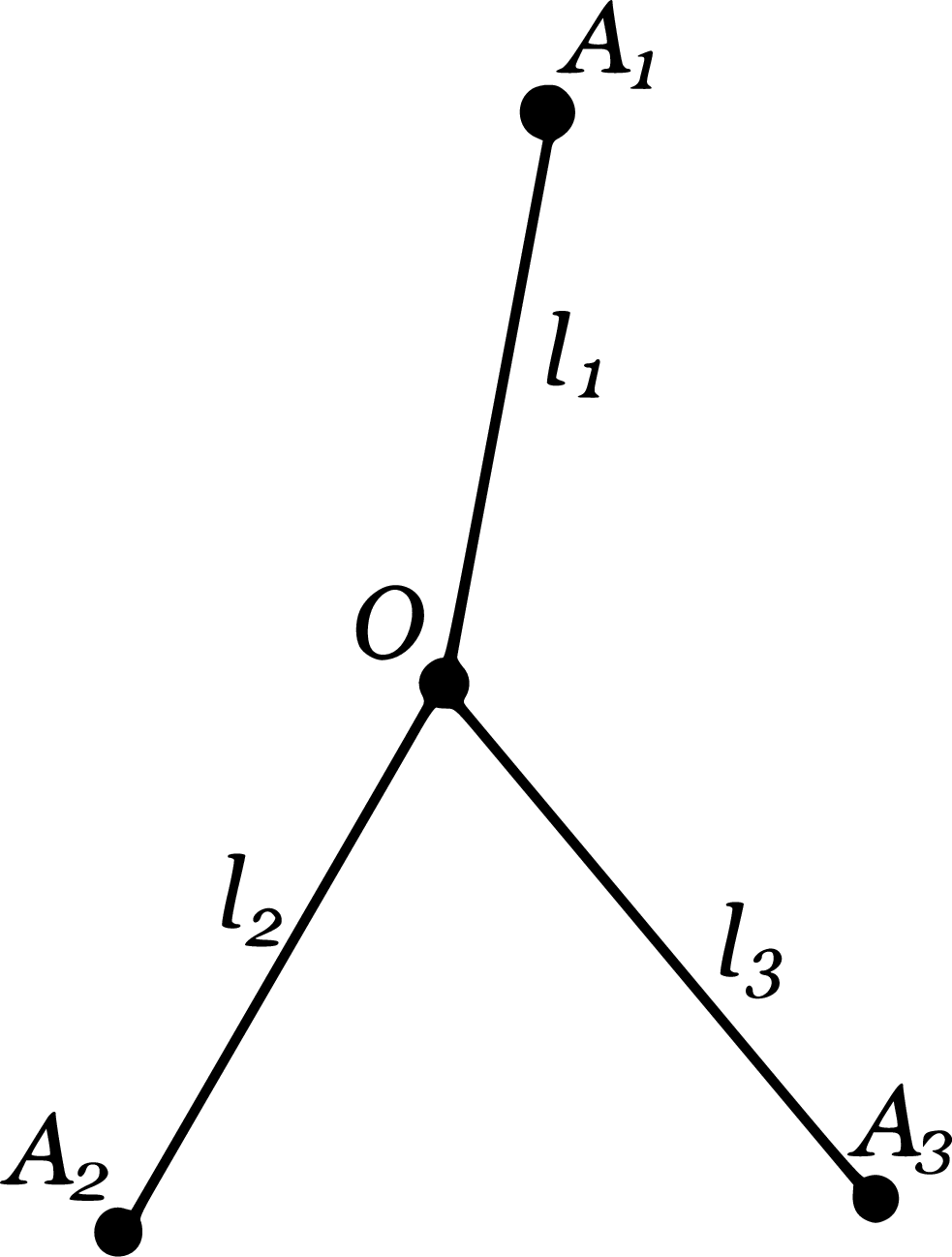}
    \qquad
    \includegraphics[scale=0.26]{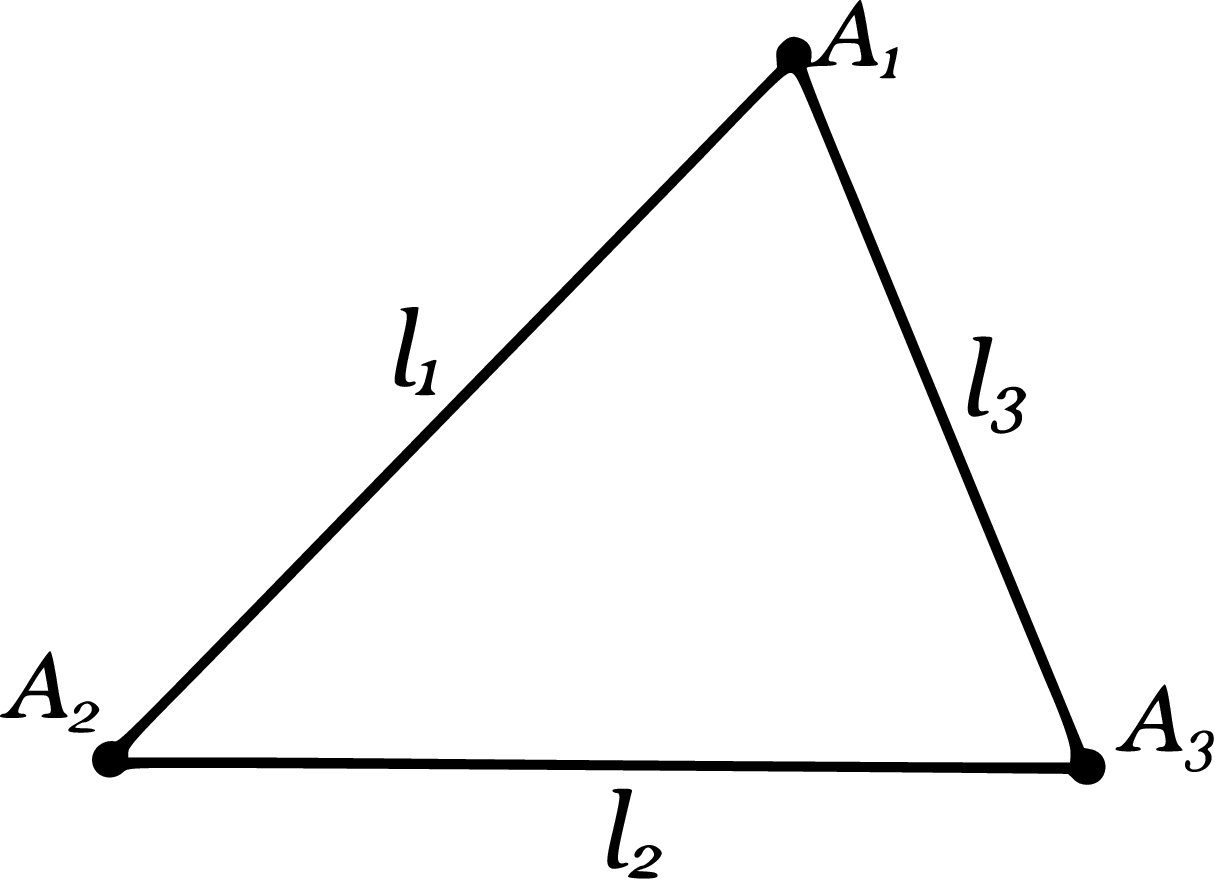}
    }
    \caption{Триод и треугольник}
    \label{f:triod}
\end{figure}

\textbf{Циклическое число} $\cy(l_1,l_2,l_3)$ циклической тройки ломаных $l_1, l_2, l_3$  определяется как число оборотов вектора в результате следующих вращений:  

$\bullet$ от вектора $\overrightarrow{A_1A_2}$ к вектору $\overrightarrow{A_1A_3}$, при котором начало вектора неподвижно, а конец движется по ломаной $l_2$, затем

$\bullet$ от вектора $\overrightarrow{A_1A_3}$ к вектору $\overrightarrow{A_2A_3}$, при котором конец вектора неподвижен, а начало движется по $l_1$, затем

$\bullet$ от вектора $\overrightarrow{A_2A_3}$ к вектору $\overrightarrow{A_2A_1}$, при котором начало вектора неподвижно, а конец движется по $l_3$, затем

$\bullet$ от вектора $\overrightarrow{A_2A_1}$ к вектору $\overrightarrow{A_3A_1}$, затем

$\bullet$ от вектора $\overrightarrow{A_3A_1}$ к вектору $\overrightarrow{A_3A_2}$, затем

$\bullet$ от вектора $\overrightarrow{A_3A_2}$ к вектору $\overrightarrow{A_1A_2}$.

Это равно удвоенному (нецелому) числу оборотов в результате первых трех из вышеприведенных вращений.
Строго говоря,
$$\cy(l_1,l_2,l_3) := w'(l_2, A_1)+ w'(l_1, A_3)+ w'(l_3, A_2)+w'(l_2, A_1)+ w'(l_1, A_3)+w'(l_3, A_2)=$$ 
$$=2\left(w'(l_2, A_1)+w'(l_1, A_3)+w'(l_3, A_2)\right). \qquad(\text{cy})$$

\begin{example}\label{e:cy}
    (a) Циклическое число циклической тройки ломаных, образующих треугольник, равно $\pm1$ (рисунок \ref{f:triod}, справа).

     (b) Циклическое число циклической тройки ломаных, образующих замкнутую несамопересекающуюся ломаную, равно $\pm1$ (это утверждение близко к теореме Жордана). 
\end{example}

\begin{proposition}\label{pr:off}
     Циклическое число любой циклической тройки ломаных нечетно. 
\end{proposition}

\begin{proof}
Обозначим через $l_1, l_2, l_3$ ломаные, образующие циклическую тройку, и соединяющие $A_1$ с $A_2$, $A_2$ с $A_3$, $A_3$ с $A_1$ соответственно.
Имеем 
$$\cy(l_1, l_2, l_3)\stackrel{(1)}{=}
2\left( \frac{\angle A_2A_1A_3}{2\pi}+k_1+ \frac{\angle A_3A_2A_1}{2\pi}+k_2+ \frac{\angle A_1A_3A_2}{2\pi}+k_3 \right)=$$
$$=2(k_1+k_2+k_3)+2\frac{\angle A_2A_1A_3+\angle A_3A_2A_1+\angle A_1A_3A_2}{2\pi}=2(k_1+k_2+k_3)+1$$
для некоторых целых $k_1, k_2, k_3$.
Здесь равенство (1) следует из определения (cy) и из утверждения~\ref{p:w'}.
\end{proof}

\begin{example}\label{p:off} 
     Для любого $n$ существует циклическая тройка ломаных, циклическое число которой равно $2n+1$. (См. рисунок \ref{f:triod5}, справа, для $n = 2$.)
\end{example} 

\begin{proof}[Построение]
Рассмотрим правильный треугольник $A_1A_2A_3$, вершины которого пронумерованы против часовой стрелки. 
Обозначим через $l_2', l_3, l_1$ ломаные $A_2A_3$, $A_3A_1$, $A_1A_2$ соответственно.
Обозначим через $l_2$ ломаную, полученную в результате $|n|$ положительных/отрицательных пальцевых движений (рисунок~\ref{f:fingermove}) ломаной $l_2'$ вокруг вершины $A_1$ для положительного/отрицательного $n$ соответственно.
Тогда 
$$\cy(l_1, l_2, l_3) = 
2\left(\frac{\angle A_2A_1A_3}{2\pi}+n+\frac{\angle A_3A_2A_1}{2\pi}+\frac{\angle A_1A_3A_2}{2\pi}\right) = 2n+1.$$
\end{proof}

\begin{figure}[ht] \centering
\compile{
\includegraphics[scale=0.25]{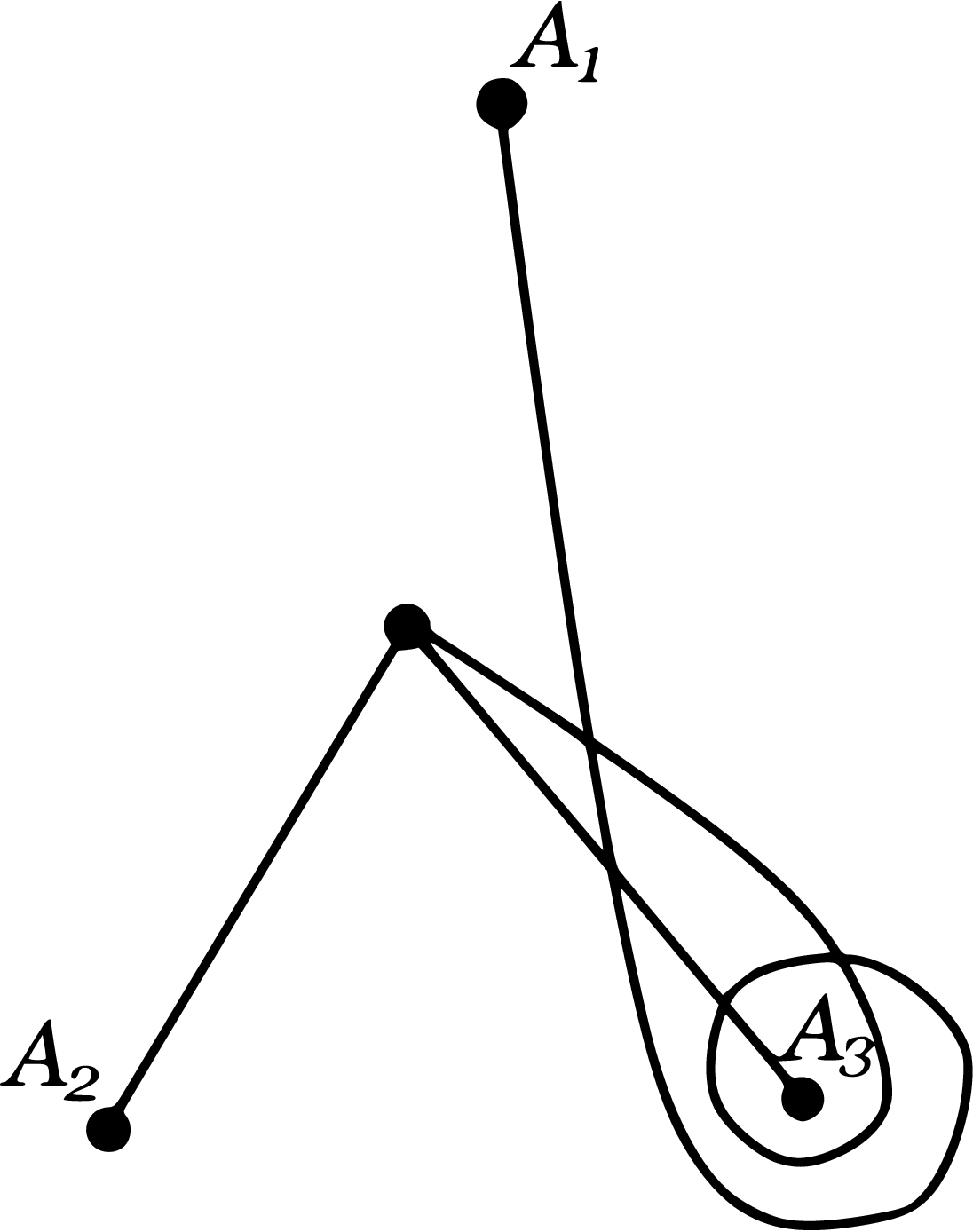}
\qquad
\includegraphics[scale=0.3]{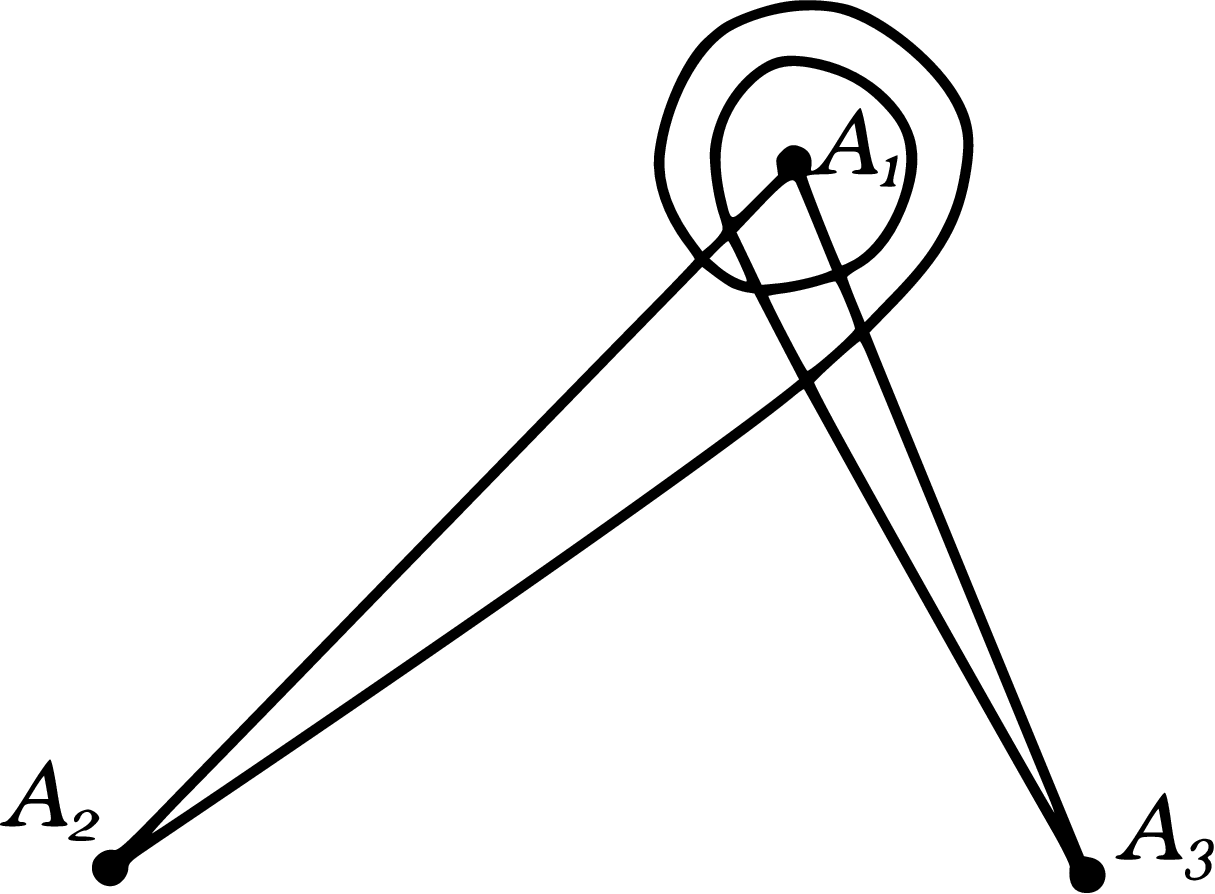}
}
\caption{Тройка ломаных, триодическое/циклическое (слева/справа) число которой равно $5$}
\label{f:triod5}
\end{figure}

Назовем тройку ломаных $l_1,l_2,l_3$, соединяющих точку $O$ с точками $A_1,A_2,A_3$ соответственно, \textbf{триодической}, если $A_i \not\in l_j$ для всех $i\neq j$.
(Другими словами, если ломаные образуют почти вложение графа $K_{3, 1}$ в плоскость.)

\textbf{Триодическое число} $\tr(l_1,l_2,l_3)$ триодической тройки ломаных $l_1, l_2, l_3$ определяется как число оборотов вектора в результате следующих вращений: 

$\bullet$ от вектора $\overrightarrow{A_1A_2}$ к вектору $\overrightarrow{A_1A_3}$, при котором начало вектора неподвижно, а конец движется по ломаной $l_2^{-1} l_3$, затем

$\bullet$ от вектора $\overrightarrow{A_1A_3}$ к вектору $\overrightarrow{A_2A_3}$, при котором конец вектора неподвижен, а начало движется по $l_1^{-1}l_2$, затем

$\bullet$ от вектора $\overrightarrow{A_2A_3}$ к вектору $\overrightarrow{A_2A_1}$, при котором начало вектора неподвижно, а конец движется по $l_3^{-1} l_1$, затем

$\bullet$ от вектора $\overrightarrow{A_2A_1}$ к вектору $\overrightarrow{A_3A_1}$ по $l_2^{-1} l_3$, затем

$\bullet$ от вектора $\overrightarrow{A_3A_1}$ к вектору $\overrightarrow{A_3A_2}$ по $l_1^{-1} l_2$, затем

$\bullet$ от вектора $\overrightarrow{A_3A_2}$ к вектору $\overrightarrow{A_1A_2}$ по $l_3^{-1} l_1$.

Это равно удвоенному (нецелому) числу оборотов в результате первых трех из вышеприведенных вращений. 
Строго говоря, $\tr(l_1, l_2, l_3)$ определяется следующей формулой, аналогичной (cy): 
$$\tr(l_1,l_2,l_3) := 
  w'(l_2^{-1}l_3, A_1) + w'(l_1^{-1}l_2, A_3) + w'(l_3^{-1}l_1,A_2) + 
w'(l_2^{-1} l_3, A_1) + w'(l_1^{-1} l_2, A_3) + w'(l_3^{-1} l_1, A_2) =$$ 
$$2\left(w'(l_2^{-1} l_3, A_1)+w'(l_1^{-1} l_2, A_3)+w'(l_3^{-1} l_1, A_2)\right). \qquad(\text{tr})$$

\begin{figure}[ht]\centering 
\compile{
    \includegraphics[scale=0.7]{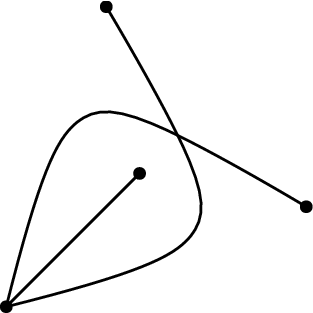}
    }
    \caption{Тройка ломаных с триодическим числом $\pm3$}
    \label{f:k31}
\end{figure}

\begin{example}[ср. пример \ref{e:cy}]\label{e:triod}
    (a) Триодическое число трех отрезков, соединяющих вершины правильного треугольника с его центром, равно $\pm1$ (рисунок~\ref{f:triod}, слева). 

    (b) Триодическое число тройки ломаных на рисунке \ref{f:k31} равно $\pm3$. 
    
    (c) Для любого вложения $f:K_{3,1}\to\R^2$ триодическое число равно $\pm1$. 
    Это доказывается индукцией по количеству вершин в ломаных из определения вложения графа в плоскость. 
    
    (d) Перестановка ломаных триодической тройки ломаных умножает их триодическое число на знак перестановки.
\end{example}

\begin{proposition}[ср. утверждение \ref{pr:off}] \label{pr:triod}
    Триодическое число любой триодической тройки ломаных  нечетно. 
\end{proposition}

\begin{proof}
    Обозначим через $l_1,l_2,l_3$ ломаные, образующие триодическую тройку, и соединяющие точку $O$ с точками $A_1,A_2,A_3$ соответственно.
    Имеем
    $$\tr(l_1, l_2, l_3) \stackrel{(1)}{=} 
    2\left( \frac{\angle A_2A_1A_3}{2\pi}+k_1 +\right.\left. \frac{\angle A_3A_2A_1}{2\pi}+k_2 + \frac{\angle A_1A_3A_2}{2\pi}+k_3\right)=$$
    $$=2(k_1+k_2+k_3)+2\left(\frac{\angle A_2A_1A_3+\angle A_3A_2A_1+\angle A_1A_3A_2}{2\pi}\right)=2(k_1+k_2+k_3)+1$$
    для некоторых целых $k_1, k_2, k_3$. 
    Здесь равенство (1) следует из определения (tr) и утверждения \ref{p:w'}.
\end{proof}

\begin{example}[ср. пример \ref{p:off}]\label{p:triod}
    Для любого целого $n$ существует триодическая тройка ломаных, триодическое число которой равно $2n+1$. (См. рисунок \ref{f:triod5}, слева, для $n = 2$.)
\end{example}

\begin{proof}[Построение]
Возьмем правильный треугольник $A_1A_2A_3$, вершины которого пронумерованы против часовой стрелки.
Обозначим через $l_1', l_2, l_3$ отрезки, соединяющие вершины $A_1,A_2,A_3$ треугольника с его центром $O$ соответственно.  
Обозначим через $l_1$ ломаную, полученную в результате $|n|$ положительных/отрицательных пальцевых движений (рисунок~\ref{f:fingermove}) ломаной $l_1'$ вокруг вершины $A_3$ для положительного/отрицательного $n$ соответственно. 
Тогда 
$$\tr(l_1, l_2, l_3)=2\left( \frac{\angle A_2A_1A_3}{2\pi}+\frac{\angle A_3A_2A_1}{2\pi}+\frac{\angle A_1A_3A_2}{2\pi}+n\right)=1+2n.$$ 
\end{proof}

\begin{remark}\label{r:cyc}
    (a) \emph{Степенью} замкнутой гладкой кривой называется число оборотов касательного вектора при обходе кривой. 
    \emph{Степенью} замкнутой ломаной $A_1\ldots A_m$ называется число 
    $$\frac{(\pi-\angle A_1A_2A_3)+(\pi-\angle A_2A_3A_4)+\ldots+(\pi-\angle A_{m-1}A_mA_1)+(\pi-\angle A_mA_1A_2)}{2\pi}.$$ 
    Циклическое число схоже со степенью замкнутой ломаной, но отличается от нее.  

    (b) Триодическое число 
    можно выразить через циклическое так:
    если $l_1, l_2, l_3$ --- триодическая тройка ломаных, то $l_1^{-1}l_2, l_2^{-1}l_3, l_3^{-1}l_1$ --- циклическая тройка ломаных и 
    $$\tr(l_1, l_2, l_3) = \cy(l_1^{-1}l_2, l_2^{-1}l_3, l_3^{-1}l_1).$$

    (c) Выберем произвольную точку $O$.
    Обозначим через $l + B$ и $l-B$ ломаные, полученные из ломаной $l$ параллельным переносом на векторы $\overrightarrow{OB}$ и $\overrightarrow{BO}$ соответственно.
    Обозначим через $-l$ ломаную, полученную из ломаной $l$ центральной симметрией относительно точки $O$.
    
    Тогда циклическое число циклической тройки ломаных $l_1 = A_1 \ldots A_2,~ l_2 = A_2 \ldots A_3,~ \linebreak l_3 = A_3 \ldots A_1$ равно числу оборотов вокруг точки $O$ замкнутой ломаной $l$, полученной конкатенацией ломаных $l_2 - A_1,~ -l_1 + A_3,~ l_3 - A_2,~ -l_2 + A_1,~ l_1 - A_3,~ -l_3 + A_2$. 

    Теорема~\ref{t:borsuk} Борсука--Улама, примененная к замкнутой ломаной $l$, дает альтернативное доказательство утверждения \ref{pr:off} (нечетности циклического числа).
    Аналогично доказывается утверждение \ref{pr:triod} (нечетность триодического числа).
   
    \aronly{(d) Возьмем циклическую тройку ломаных $l_1, l_2, l_3$.
    Представим ломаную $l_i$ в качестве \emph{PL отображения} $p_i : [0, 1] \to \R^2$ для каждого $i = 1, 2, 3.$
    Тогда циклическое число для тройки $l_1, l_2, l_3$ равно числу оборотов вектора $p_1(x) + p_2(y) + p_3(z) - p_1(0) - p_2(0) - p_3(0)$ при обходе по следующему циклу, идущему по ребрам куба $[0, 1]^3$: 
    $$(0, 0, 1) \to (0, 1, 1) \to (0, 1, 0) \to (1, 1, 0) \to (1, 0, 0) \to (1, 0, 1) \to (0, 0, 1).$$ 
    
    (e) Возьмем любые два непрерывных отображения $f, g : S^1 \to \R^2$, такие что $f(s) \neq g(s)$ для каждой точки $s$ окружности $S^1$.
    Обозначим через $w(f-g)$ число оборотов, которое совершает ненулевой вектор $f(s) - g(s)$ при обходе окружности против часовой стрелки.
    Тогда для любой
    
    $\bullet$ точки и замкнутой ломаной, не проходящей через нее,

    $\bullet$ циклической тройки ломаных,

    $\bullet$ триодической тройки ломаных

    найдутся такие $f$ и $g$, что число $w(f-g)$ равно числу оборотов, циклическому и триодическому числу соответственно.
    (Аналогичное утверждение верно и для числа $\partial(l \times p)$, определенного перед леммой \ref{l:square}.)
    }
\end{remark}

\section{Связь инвариантов друг с другом}\label{s:gawhap}

Для слабого почти вложения $f:K\to\R^2$ и $K_{3, 1}$-подграфа в $K$ с долями $i,j,k$ и $v$ обозначим
$$\tr_f(ijk, v) := \tr(f|_{vi}, f|_{vj}, f|_{vk}).$$

Для слабого почти вложения $f:K\to\R^2$ и простого цикла $ijk$ в $K$ обозначим
$$\cy_f(ijk) := \cy(f|_{ij}, f|_{jk}, f|_{ki}).$$

Для $K = K_4$ 
обозначим 
$$\tr_f(v) := \tr_f(C_v, v) \quad \text{и} \quad \cy_f(v):=\cy(C_v).$$

\begin{proposition}\label{p:wu-conj0} 
Для любого слабого почти вложения $f : K_4 \to \R^2$ имеем
$$2W_f = \cy_f(1) - \cy_f(2) + \cy_f(3) - \cy_f(4).$$
\end{proposition}

\begin{proposition}\label{p:wu-conj} 
Для любого почти вложения $f:K_4\to\R^2$ имеем

(a) $W_f = \tr_f(4) = -\tr_f(3) = \tr_f(2) = -\tr_f(1)$;

(b) $2w_f(j) = \cy_f(j) + (-1)^j\tr_f(j)$ для каждого $j = 1, 2, 3, 4$.
\end{proposition}

Доказательства утверждений \ref{p:wu-conj0} и \ref{p:wu-conj} приведены в \S\ref{s:gawhap}.A. 

Нечетность числа $W_f$ для почти вложения $f:K_4\to\R^2$ (теорема \ref{t:radonae}) вытекает из утверждения \ref{p:wu-conj}.a и нечетности триодического числа (утверждение \ref{pr:triod}).
(Это доказательство нечетности числа $W_f$ дает альтернативное доказательство топологической теоремы Радона \ref{t:radon} для плоскости с заменой отображения на почти вложение.)

Утверждение \ref{p:wu-conj} позволяет выразить все триодические и циклические числа через оборотные, а вместе с утверждением \ref{p:wu-conj0} --- все триодические и оборотные через циклические.



\begin{remark}[соотношения между циклическими числами почти вложения графа $K_4$]
\label{r:cy-k4}
    Из теоремы~\ref{t:radonae} и утверждения~\ref{p:wu-conj0} следует, что \emph{для любого почти вложения $f:K_4\to\R^2$ верно}
    $$\cy_f(1) - \cy_f(2) + \cy_f(3) - \cy_f(4) \equiv 2 \pmod 4.$$
    Ввиду теоремы~\ref{t:k4} и утверждения~\ref{p:wu-conj}, между этими четырьмя циклическими числами нет других соотношений.
    \aronly{Более точно, 
    \emph{для любых нечетных чисел $m_1, m_2, m_3, m_4$, для которых $m_1 - m_2 + m_3 - m_4 \equiv 2 \pmod 4$, существует почти вложение $f : K_4 \to \R^2$, такое что $\cy_f(j) = m_j$ для каждого $j = 1,2,3, 4$}.

    \emph{Доказательство.}
    Обозначим $W := (m_1 - m_2 + m_3 - m_4)/2$.
    По теореме~\ref{t:k4} существует почти вложение $f : K_4 \to \R^2$, такое что $2w_f(j) = m_j + (-1)^jW$.
    Оно и является искомым.
    Действительно, для каждого $j = 1, 2, 3, 4$ имеем
    $$\cy_f(j) \stackrel{(1)}{=} 2w_f(j) + (-1)^{j-1}\tr_f(j) = m_j + (-1)^jW + (-1)^{j-1}\tr_f(j)\stackrel{(2)}{=}  m_j + (-1)^j(W - W_f) \stackrel{(3)}{=} m_j.$$
    Здесь равенство (1) верно по утверждению \ref{p:wu-conj}.b, (2) --- по утверждению \ref{p:wu-conj}.а, (3) --- поскольку $W_f = W$ как следствие определений чисел $W_f$ и $W$.
    }
\end{remark}


\begin{remark}[соотношения между инвариантами почти вложения графа $K_5 - 45$]
\label{r:relations-k5-e}
(a) Для отображения $f : K_5 - 45 \to \R^2$ обозначим через $f_j$ сужение отображения $f$ на $K_4$-подграф, содержащий вершину $j = 4, 5$.
Напомним, что для почти вложения $f : K_5 - 45 \to \R^2$ мы обозначаем $W_f := W_{f_4}$.
Из утверждения~\ref{p:k4-k5} следует (см. замечание~\ref{r:k5-e}), что 
\emph{для любого почти вложения $f:K_5 - 45 \to \R^2$ верно}
$$W_f := W_{f_4} = -W_{f_5}.$$

\begin{figure}[ht]\centering
\compile{
    \includegraphics[scale=0.9]{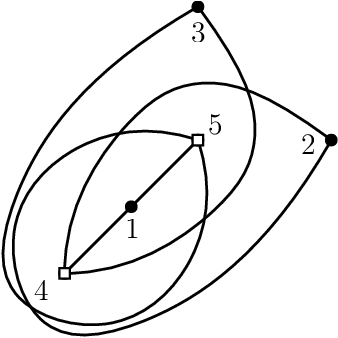}
    }
\caption{Почти вложение $f : K_{3,2} \to \R^2$, для которого $\tr_f(123, 4)=\tr_f(123, 5) \neq -\tr_f(123, 5)$}
    \label{f:k23-conj}
\end{figure}

(b) Из пункта (a) и утверждения \ref{p:wu-conj}.a следует, что 
\emph{для любого почти вложения $f : K_5 - 45 \to \R^2$ верно}
$$W_f = \tr_f(123, 4) = -\tr_f(123, 5).$$

Ср. рисунок~\ref{f:k23-conj}.

(c) Из пункта (b) и утверждения \ref{p:wu-conj}.b следует, что 
\emph{для любого почти вложения $f : K_5 - 45 \to \R^2$ верно}
$$\cy_f(123) = w_f(123, 4) + w_f(123, 5).$$
\end{remark}

\section*{\ref{s:gawhap}.A.
Доказательства утверждений \ref{p:wu-conj0} и \ref{p:wu-conj}} 

Для слабого почти вложения $f:K \to \R^2$, ориентированного цикла $C$, ориентированного ребра $ij$ и вершины $v$, не лежащей ни на ребре $ij$, ни на цикле $C$, обозначим
$$C \times_f v := w_f(C, v) \quad \text{ и } \quad  ij \times_f v  := w'(f|_{ij}, f(v)).$$
Смысл новых обозначений объяснен в замечании \ref{r:delpro-mot} и \S\ref{s:Wu}.  

Доказательство утверждения \ref{p:wu-conj0} просто получается из определения числа $W_f$, а также следующих равенств: 

$\bullet$  $ij \times_f v + ji \times_f v = 0$ для любых ориентированного ребра $ij$ и вершины $v \not \in ij$;

$\bullet$ $ijk \times_f v = ij \times_f v + jk \times_f v + ki \times_f v$, если $\{i,j,k,v\}=[4]$;

$\bullet$ $\cy_f(v)=2(ij \times_f k + ki\times _f j + jk \times_f i)$, если $\{i,j,k, v\} = [4]$.

Первое равенство очевидно. 
Второе и третье --- переписанные определения.  
Эти равенства и аналог третьего из них для триодического числа также будут использоваться (без их упоминания) в доказательстве утверждения \ref{p:wu-conj}. 

Следующее обозначение и лемма пригодятся в доказательстве утверждения~\ref{p:wu-conj}.
Для ломаных $A\ldots B$ и $C\ldots D$ обозначим 
$$\partial(A\ldots B \times C\ldots D) := w'(A\ldots B,D)+w'(C\ldots D,A)-w'(C\ldots D,B)-w'(A\ldots B,C).$$ 
Это число определено, когда ни один из концов каждой ломаной не лежит на объединении звеньев другой ломаной. 
(Представим ломаные $A\ldots B$ и $C\ldots D$ в качестве \emph{PL отображений} $p,q:[0,1]\to\R^2$.
Тогда $\partial(A\ldots B \times C\ldots D)$ есть число оборотов вектора $p(x)-q(y)$ при обходе границы $\partial([0,1]^2)$ квадрата  $[0,1]^2$ по часовой стрелке.)

\begin{lemma}\label{l:square} 
Для любых непересекающихся ломаных $l$ и $p$ выполнено $\partial(l\times p)=0$. 
\end{lemma}

Доказательство этой классической леммы см. в \cite[Лемма 2.5.a]{ABM+}.

\begin{figure}[ht]\centering
    \includegraphics[scale=1.2]{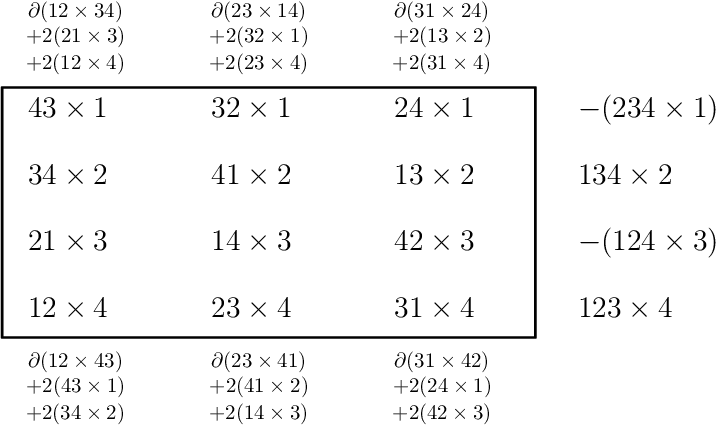}
    \caption{Различные группировки слагаемых в сумме $W_f$}
    \label{f:table}
\end{figure}

\begin{proof}[Доказательство утверждения \ref{p:wu-conj}]

    Для $\{i, j, p, q\} = [4]$ обозначим 
    $$\partial_f(ij\times pq) := ij \times_f p  + pq \times_f j + ji \times_f q + qp \times_f i.$$

    Поскольку $f$ является почти вложением, то по лемме \ref{l:square} имеем $\partial_f(ij \times pq) = 0$, если $\{i, j, p, q\} = [4]$.
   
    Первое равенство пункта (a) доказывается так:
    $$W_f \stackrel{(1)}{=}$$ 
    $$= 2\left(24 \times_f 1 + 43 \times_f 1 + 14 \times_f 3 +\right. \left. 42 \times_f 3 + 34 \times_f 2 + 41 \times_f 2\right) +$$ 
    $$+ \partial_f(12\times43) + \partial_f(23\times41) + \partial_f(31\times42) =$$ 
    $$= \tr_f(4).$$

    Для доказательства равенства (1) посмотрим на рисунок \ref{f:table} (на котором индексы $f$ удалены).
    Вне прямоугольника написана сумма в соответствующей строке или в соответствующем столбце чисел в прямоугольнике.
Сумма всех чисел в прямоугольнике равна сумме чисел справа от него и равна $W_f$.
Она также равна сумме чисел снизу от него, т.~е. правой части равенства (1).

Аналогично доказывается, что $W_f$ равно другим триодическим числам из пункта (а) (берем некоторые числа сверху, а некоторые снизу от прямоугольника).
Это следует также из кососимметричности (утверждения \ref{p:autom} и \ref{e:triod}.d). 

Аналогично доказывается равенство $W_f = 2w_f(j)-\cy_f(j)$ (например, для $j=4$ берем числа только сверху от прямоугольника). 
Из него следует пункт (b). 
Иными словами и напрямую, для $j = 4$ имеем    
$$\cy_f(4) + \tr_f(4) - 2w_f(4) \stackrel{(2)}{=} 
2\left(\partial_f(12\times 34)-\partial_f(13\times 24) -\partial_f(14\times 23)\right) = 0.$$
(Как можно придумать равенства (1) и (2), написано в замечании \ref{r:delpro-mot} и \S\ref{s:Wu}.)
\end{proof}


\begin{remark}\label{r:delpro-mot} 
(a) При обосновании равенств (1) и (2) из доказательства утверждения~\ref{p:wu-conj} (а также других равенств, иллюстрируемых рисунком~\ref{f:table}) не обязательно помнить ни о сумме углов, через которую определяются оборотное, циклическое и триодическое числа, ни об отображении $f$. 
Упростим эти равенства, убрав индексы $f$.
Это упрощение --- проявление того, что за этими равенствами стоят соотношения между \emph{целочисленными одномерными циклами} во \emph{взрезанном квадрате} графа $K_4$, см. \S\ref{s:Wu}.
Эти соотношения можно применять не только к доказательству утверждений типа \ref{p:wu-conj}, но и к их придумыванию (ибо вместо формул можно смотреть на картинки, см. п. (b)). 
Эти соотношения можно применять и к другим задачам, см. ссылки в \cite[Предисловие]{MNS}.

\begin{figure}[ht] \centerline{
    \includegraphics[width=6cm]{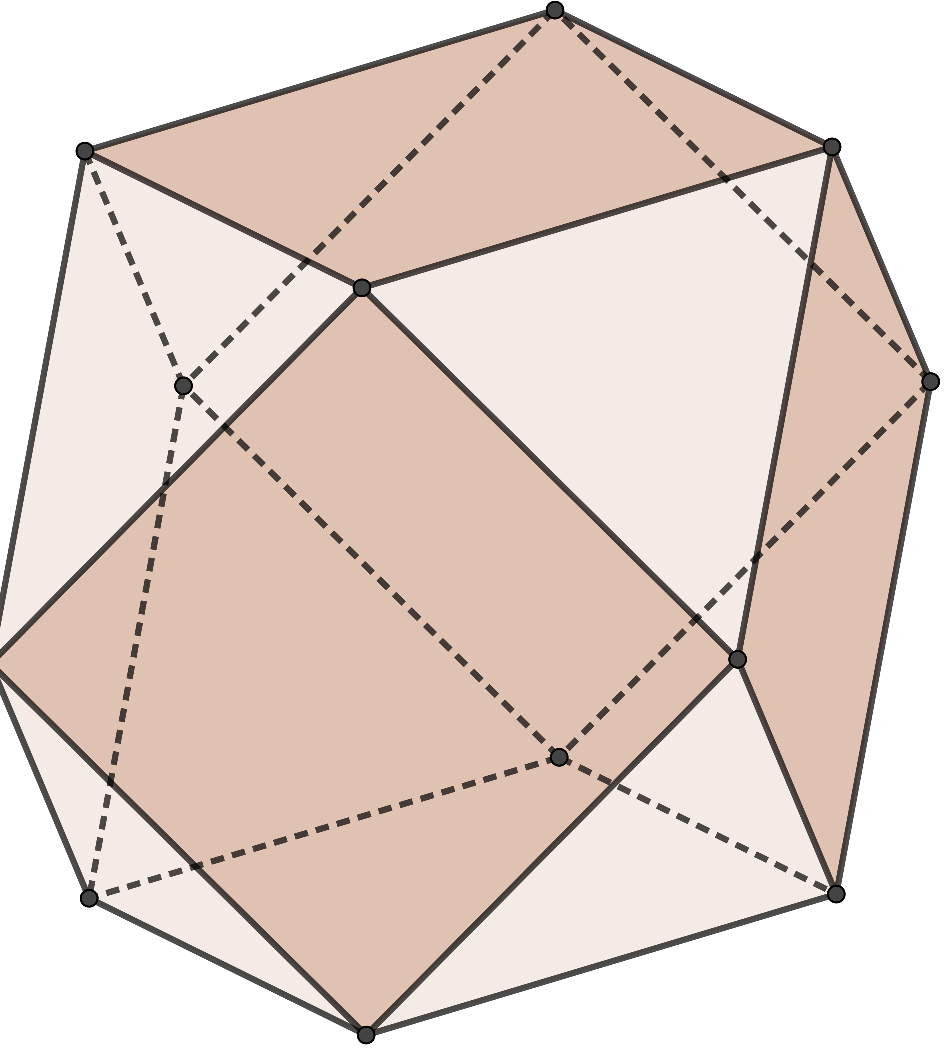}
    \qquad
    \includegraphics[width=7cm]{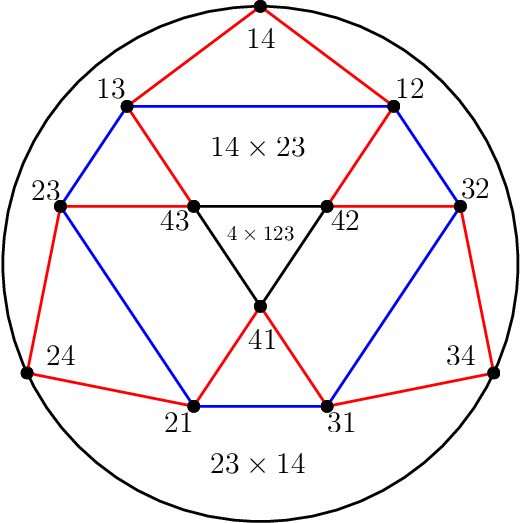}
    }
    \caption{Слева: кубооктаэдр --- взрезанный квадрат графа $K_4$ (точнее, \emph{симплициальный взрезанный квадрат}, см. определение в \cite[\S8.2]{Sk} или \cite[\S4, The van Kampen obstruction mod 2]{Sk06}, графа $K_4$).
    Справа: взрезанный бокс-квадрат графа $K_4$ (вершины $i \times j$ сокращаются до $ij$).} 
    \label{f:del3}
\end{figure} 

(b) (здесь мы неформально объясняем, как придумать и геометрически интерпретировать утверждение \ref{p:wu-conj}; см. формализацию в \S\ref{s:Wu})
На правом рисунке \ref{f:del3} (ср. \cite[\S1.9]{MNS}) 

$\bullet$ $4\times 123$ "--- ориентированный контур центрального треугольника, 

$\bullet$ $123\times4$ "--- ориентированная окружность,

$\bullet$ $\partial(14\times 23)$ "--- ориентированная граница центральной верхней трапецевидной грани $14\times 23$, 

$\bullet$ $\partial(23\times14)$ "--- ориентированная граница центральной нижней криволинейной грани $23\times14$, 

$\bullet$ ориентированные циклы $\adiag(123)$ и $\triod(123, 4)$ выделены синим и красным соответственно.

Каждый из циклов  $\adiag(123)$ и $\triod(123, 4)$ разбивает кубооктаэдр на две равные части.

Магическая формула для равенства (1) из доказательства утверждения~\ref{p:wu-conj}:
    $$\triod(123,4) + \partial(12\times 43)+\partial(23\times 41)+\partial(31\times 42)=
    -234\times 1 + 134\times 2 - 124\times 3 + 123\times 4.$$
Магическая формула для равенства (2):
    $$\adiag(123) + \triod(123,4) =$$ 
    $$123\times 4 + 4\times 123 + \partial(31\times 24) + \partial(24\times 13) + \partial(12\times 34) + \partial(34\times 21) + \partial(14\times 32) + \partial(23\times 14).$$
\end{remark}

\section{Более общие инварианты: числа Ву}
\label{s:Wu}

\begin{remark}\label{r:noalis}
(a) Аналог циклического числа можно определить для замкнутой ломаной, составленной из нескольких ломаных $l_1,l_2,\ldots,l_n$, образующих слабое почти вложение графа-цикла.
Например, для $n=4$ можно взять
$$\cy(l_1,l_2,l_3,l_4) := w’(l_2l_3, A_1) + w’(l_3l_4, A_2) + w’(l_4l_1, A_3) + w’(l_1l_2, A_4).$$

Если ломаные образуют почти вложение графа-цикла, то
$$\cy(l_1,l_2,l_3,l_4) = \cy(l_1l_2, l_3, l_4) = \cy(l_1, l_2l_3, l_4) = \cy(l_1, l_2, l_3l_4) = \cy(l_4l_1, l_2, l_3)$$
ввиду леммы \ref{l:square}, примененной к парам сужений отображения на несмежные ребра цикла.

(b) Для любого почти вложения $f: K_4 \to \R^2$ имеем 
$$\cy(f|_{12}, f|_{23}, f|_{34}, f|_{41}) = \sum_{j = 1}^4w_f(C_j, j).$$
(Ср. утверждение \ref{p:wu-conj}.a.)
Это равенство получается применением леммы \ref{l:square} к паре сужений отображения на ребра 13 и 24.
Вместе с равенствами из п. (а) и утверждением \ref{pr:off} оно дает еще одно доказательство теоремы \ref{t:radonae}.

(c) Аналоги триодического числа можно определить для $n$ ломаных с общей вершиной, образующих слабое почти вложение графа $K_{n,1}$.
\end{remark}

\emph{Взрезанным бокс-квадратом} $K^{\Box \underline 2}$ графа $K$ называется граф,

$\bullet$ вершинами которого являются упорядоченные пары $i \times j$ \emph{различных} вершин $i, j$ графа $K$;

$\bullet$ ребро между вершинами $i \times p$ и $j \times q$  проводится, если либо $i = j$ и $pq$ --- ребро в $K$, либо $ij$ --- ребро в $K$ и $p = q$; эти ребра обозначаются $i \times pq$ и $ij \times p$ соответственно.

\begin{remark}[определение взрезанного квадрата графа\aronly{ и клеточной структуры на его подмножестве}]
\label{r:delpro-def}
Множество $\widetilde K$ пар различных точек (тела\footnote{Рассмотрим произвольный симплекс, вершины которого отождествлены с вершинами графа~$K$. Тогда \emph{телом} графа~$K$ называется объединение тех ребер симплекса, которые соответствуют ребрам графа~$K$.}) графа $K$
называется 
\textit{взрезанным квадратом} графа $K$, т.~е.
$
\widetilde K := \{ (x, y) \in K \times K : x \neq y \}.
$
(Более точно, в предыдущей формуле нужно заменить $K \times K$ на $|K| \times |K|$, где $|K|$ "--- тело графа $K$.)

\aronly{
Далее, под вершинами и ребрами графа~$K$ мы также подразумеваем соответствующие подмножества тела графа~$K$.

Во взрезанном квадрате графа $K$ 
имеется естественное подмножество
$$\cup \{ \sigma\times\tau\ : \sigma, \tau\ \textrm{---  симплексы в}\ K,\ \sigma\cap\tau = \emptyset\},$$
называемое \textit{симплициальным взрезанным квадратом} графа $K$ (ср. с определениями в \cite[\S2.4]{MNS}, \cite[\S5]{Sk06}). На этом подмножестве имеется естественная клеточная структура: 

$\bullet$ $\{ v_1 \times v_2\ : v_1, v_2\textrm{ --- различные вершины в }K \}$ --- набор 0-клеток (вершин),

$\bullet$ $\{ v \times e,\ e \times v\ : v\textrm{ --- вершина в }K\text{, не принадлежащая ребру }e\text{ графа }K \}$ --- набор 1-клеток (ребер),

$\bullet$ $\{ e_1 \times e_2\ : e_1, e_2\textrm{ --- несмежные ребра в }K \}$ --- набор 2-клеток (граней).
}

Взрезанный бокс-квадрат $K^{\Box \underline 2}$ (более точно, его тело) является объединением ребер взрезанного квадрата $\widetilde K$ графа $K$ (более точно, 
\emph{симплициального взрезанного квадрата}; см. определение в \cite[\S8.2]{Sk} или \cite[\S4, The van Kampen obstruction mod 2]{Sk06}).
\end{remark}

Обозначим через $E$ множество ребер графа $K^{\Box \underline 2}$.
Назовем \emph{расстановкой} (чисел на ребрах графа $K^{\Box \underline 2}$) отображение $E \to \R$ (общепринятое название --- \emph{коцепь} или \emph{цепь}).
Определим \textit{сумму} расстановок $w_1, w_2$ и \textit{значение} расстановки $w$ на расстановке $C$ формулами
$$(w_1 + w_2)(e) := w_1(e) + w_2(e) \quad \text{ и }\quad  w(C) := \sum_{e \in E} w(e)C(e).$$

Зафиксируем произвольный набор ориентаций на ребрах графа $K$.
Для каждой вершины $v$ графа $K$ возьмем на ребрах $v\times e$ и $e\times v$ взрезанного бокс-квадрата  $K^{\Box \underline 2}$ направления, соответствующие направлению на ребре $e$ графа $K$.

Расстановка целых чисел называется (симплициальным) \emph{целочисленным 1-циклом} (в $K^{\Box \underline 2}$), если для каждой вершины сумма целых на входящих ребрах равна сумме чисел на исходящих (правило Кирхгофа).  

Сопоставим ориентированному циклу в $K^{\Box \underline 2}$ следующий целочисленный 1-цикл. 
Если ребро $e$ графа $K^{\Box \underline 2}$ не принадлежит ориентированному циклу, то поставим на $e$ число 0.
Иначе если ориентация, зафиксированная на ребре $e$, совпадает с ориентацией цикла, то поставим на $e$ число $+1$, а если противоположна, то $-1$.

Вот примеры ориентированных циклов в $K^{\Box \underline 2}$ (которым ввиду предыдущего абзаца сопоставляются соответствующие целочисленные 1-циклы):

$\bullet$ если $C = v_1\ldots v_n$ --- ориентированный цикл в графе $K$ и $v \not \in C$, то \emph{левым} и \emph{правым} циклом соответственно называются
$$C \times v := v_1 \times v,~\ldots, v_n \times v \quad \text{и} \quad v \times C := v \times v_1,~\ldots, v \times v_n;$$ 

$\bullet$ \emph{границей} для ориентированных несмежных ребер $ij, pq$ графа $K$ называется
$$\partial(ij \times pq) := i\times p,~ j \times p,~ j \times q,~ i \times q;$$ 

\begin{figure}[ht] \centerline
    \compile{
    \includegraphics[scale = 0.9]{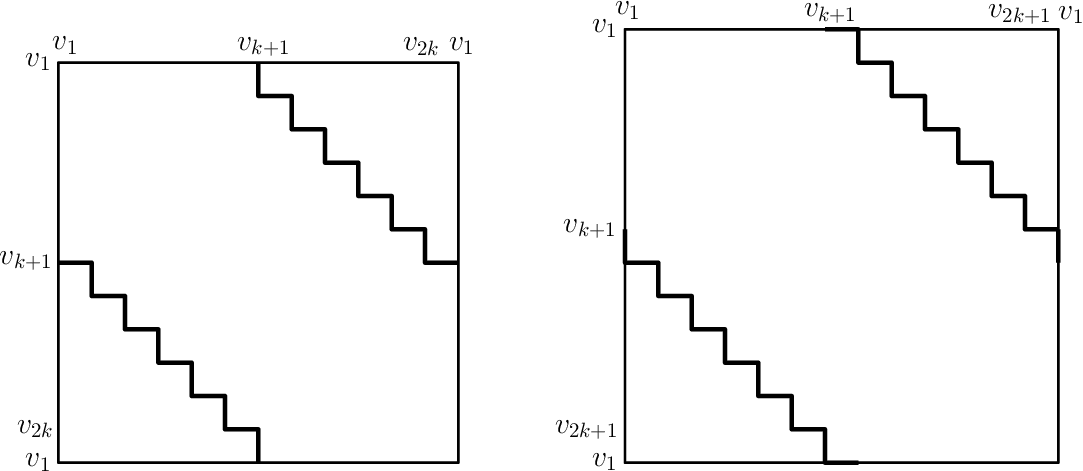}
    }
    \caption{Антидиагональный цикл, соответствующий циклу четной (слева) и нечетной (справа) длины для $k = 6$}
    \label{f:adiag}
\end{figure}

$\bullet$ для $K = K_3$ \emph{антидиагональным циклом} называется
$$\adiag(123) := 1\times2,~ 1\times3,~ 2\times3,~ 2\times 1,~ 3 \times 1,~ 3 \times 2;$$ 

$\bullet$ если $v_1\ldots v_{n}$ "--- ориентированный простой цикл в графе $K$, то \emph{антидиагональным циклом} (см. рисунок~\ref{f:adiag}) называется
\begin{align*}
    \adiag(v_1 \ldots v_{2k+1}) :=\ &v_1 \times v_{k+1},~ v_1 \times v_{k+2},~ v_2 \times v_{k+2},~ \ldots, ~ v_{k+1} \times v_{2k+1},\\
    &v_{k+1} \times v_{1},~ v_{k+2} \times v_{1},~ v_{k+2} \times v_{2},~ \ldots, ~ v_{2k+1} \times v_{k+1},\\
    \adiag(v_1 \ldots v_{2k}) :=\ &v_1 \times v_{k+1},~ v_2 \times v_{k+1},~ v_2 \times v_{k+2},~ \ldots, ~ v_{k+1} \times v_{2k},\\
    &v_{k+1} \times v_{1},~ v_{k+1} \times v_{2},~ v_{k+2} \times v_{2},~ \ldots, ~ v_{2k} \times v_{k+1}
\end{align*}
(этот цикл является дискретной версией цикла $\{(x, -x)\ :\ x\in S^1\}\subset S^1 \times S^1$ с некоторой ориентацией; он гомологичен околодиагональному циклу из \cite[\S1.7, перед задачей 1.7.2]{MNS}; в статье \cite{MNS} антидиагональным называется совсем другой цикл);

$\bullet$ для $K = K_{3,1}$ \emph{триодическим циклом} называется
$$\triod(123, 4) := 1 \times 2,~ 1 \times 4,~ 1\times3,~ 4\times3,~ 2\times3,~ 2\times4,~ 2\times1,~ 4\times1,~ 3\times1,~ 3\times4,~ 3\times2,~ 4\times2$$
(ср. определение триодического цикла \cite[\S1.9, после гипотезы 1.9.3]{MNS});

$\bullet$ \emph{триодическим циклом} также называется аналогичный цикл, соответствующий $K_{3,1}$-подграфу графа $K$.

Имеем

$\bullet$ если $C_1, C_2$ --- целочисленные 1-циклы в графе $K^{\Box \underline{2}}$, то их сумма $C_1 + C_2$ тоже является целочисленным 1-циклом;

$\bullet$ если $C$ --- целочисленный 1-цикл и $\alpha$ целое, то $\alpha  C$ тоже является целочисленным 1-циклом.

Для каждого целочисленного 1-цикла $C$ и почти вложения $f$ определим \textbf{$C$-число Ву} $w_f(C)$ как значение на $C$ следующей \textit{расстановки вращений} $w_f$.
Возьмем ребро $bc$ графа $K$, ориентированное от $b$ к $c$.
Для ребер $bc\times a$ и $a\times bc$ графа $K^{\square \underline2}$ определим
$$w_f(bc\times a)=w_f(a\times bc) := w'(f|_{bc},f(a)) \in \R.$$ 
Это количество оборотов вектора с началом в $f(a)$ и концом, пробегающим ломаную $f|_{bc}$.

\begin{proposition}\label{propome0} Для любого почти вложения $f : K \to \R^2$

(a) число $w_f(C)$ целое для произвольного целочисленного 1-цикла $C$;

(b) если в графе $K^{\Box \underline{2}}$ вершины $a \times b$ и $b \times a$ соединены ориентированным путем $l$, то число
$2w_f(l)$ целое и нечетное (ср. замечание \ref{r:cyc}.c; сопоставление ориентированному пути расстановки аналогично сопоставлению ориентированному циклу 1-цикла).
\end{proposition}

\begin{proposition}\label{propome} Для любого почти вложения $f : K \to \R^2$ имеем

(a) $w_f(C \times v) = w_f(v \times C) = w_f(C, v)$ для любых ориентированного цикла $C$ в $K$ и вершины $v \not \in C$;

(b) $w_f(\partial(e \times e')) = 0$ для любых ориентированных несмежных ребер $e, e'$ графа $K$ (по лемме \ref{l:square});

(c) $w_f(\adiag(123)) = \cy_f(4)$ и $w_f(\triod(123, 4)) = \tr_f(4)$ для $K = K_4$;

(d) $w_f(C_1 + C_2) = w_f(C_1) + w_f(C_2)$ и $w_f(\alpha C) = \alpha w_f(C)$ для любых целочисленных 1-циклов $C_1, C_2, C$ и целого $\alpha$.
\end{proposition}

Из равенства целочисленных 1-циклов следует равенство их чисел Ву.
Поэтому и ввиду утверждения~\ref{propome} магические равенства из замечания~\ref{r:delpro-mot}.b доказывают равенства (1) и (2) из доказательства утверждения~\ref{p:wu-conj}.


Два целочисленных 1-цикла называются \emph{гомологичными}, если их разность равна линейной комбинации границ $\partial(e \times e')$ с целыми коэффициентами.
Если два целочисленных 1-цикла гомологичны, то соответствующие им числа Ву совпадают ввиду утверждения \ref{propome}.b.
Соотношения между оборотными, циклическими и триодическими числами в утверждениях \ref{p:k4-k5}, \ref{p:k4-k33}, \ref{p:wu-conj0}, \ref{p:wu-conj} и замечаниях \ref{r:k5-e}, \ref{r:cy-k4}, \ref{r:relations-k5-e} возникают из гомологичности некоторых целочисленных 1-циклов.
Теоремы \ref{t:radonae}, \ref{t:k5-e}.a и утверждения \ref{pr:off}, \ref{pr:triod} имеют другую (но тоже алгебраическую) природу, связанную с теоремой Борсука--Улама \ref{t:borsuk}.
Теорема \ref{t:k5-e}.b --- существенно другой, геометрический результат.


Обозначим через $H_1(\t K, \Z)$ группу всех целочисленных 1-циклов в $K^{\Box \underline{2}}$ с точностью до гомологичности.
С точностью до изоморфизма \emph{инвариант Ву} почти вложения $f$ --- это гомоморфизм $w_f : H_1(\t K, \Z)\to\Z$, см. утверждение \ref{propome}.d\aronly{ и ср. замечание \ref{r:defwu}}.

\begin{problem}[ср. проблему \ref{pm:wind}]\label{pm:Wu}  Для данного плоского графа $K$ опишите гомоморфизмы $H_1(\t K, \Z) \to \Z$, равные $w_f$ для некоторого почти вложения $f : K\to\R^2$.

Возможно, это описание не имеет красивой формулировки. 
Тогда интересен алгоритм, который по гомоморфизму $H_1(\t K, \Z) \to \Z$ выясняет, равен ли он $w_f$ для некоторого почти вложения $f:K\to\R^2$.
\end{problem}

Группа $H_1(\t K_4, \Z)$ порождается всеми левыми и правыми циклами $C_j \times j$ и $j \times C_j$, см. рисунок~\ref{f:del3}. 
Поэтому для описания гомоморфизма $H_1(\t K_4, \Z) \to \Z$ достаточно задать его значения на этих левых и правых циклах, т.~е. задать значения всех оборотных чисел $w_f(j)$. 
Это рассуждение вместе с теоремами \ref{t:radonae} и \ref{t:k4} дает решение проблемы \ref{pm:Wu} для $K = K_4$. 

\begin{theorem}\label{t:poly-Wu}
    Пусть $K$ "--- граф выпуклого многогранника в $\R^3$.
    Для любого целочисленного 1-цикла $C$ существует линейная комбинация с рациональными коэффициентами триодических и оборотных чисел, равная $C$-числу Ву $w_f(C)$ для любого почти вложения $f: K \to \R^2$.
\end{theorem}

Далее, \emph{циклическим числом} почти вложения $f: K \to \R^2$ мы называем $C$-число Ву $w_f(C)$ для антидиагонального цикла $C := \adiag(v_1 \ldots v_n)$, соответствующего простому циклу $v_1 \ldots v_n$ в графе $K$. 

\begin{theorem}\label{t:Wu}
    Пусть $K$ "--- связный граф.
    Для любого целочисленного 1-цикла $C$ существует линейная комбинация с рациональными коэффициентами циклических и триодических чисел, 
    равная $C$-числу Ву $w_f(C)$ для любого почти вложения $f: K \to \R^2$.
    
\end{theorem}

Теоремы~\ref{t:poly-Wu} и \ref{t:Wu} являются частными случаями следующих двух теорем, соответственно, для гомоморфизма $w_f: H_1(\widetilde K; \Z) \to \Z$.

\begin{theorem}[{\cite{Al26}}]
\label{t:tr-le-ri}
    Пусть $K$ "--- граф выпуклого многогранника в $\R^3$.
    Обозначим через $C_1, \ldots, C_n$ все триодические, левые и правые циклы.
    Для любого целочисленного 1-цикла $C$ существует набор чисел $c_1, \ldots, c_n \in \Q$, такой что для любого гомоморфизма $\varphi: H_1(\widetilde K; \Z) \to \Z$ верно
    $\varphi(C) = \sum_{i = 1}^n c_i \varphi(C_i).$
\end{theorem}

\begin{theorem}[{\cite{Al26}}]
\label{t:tr-ad}
    Пусть $K$ "--- связный граф.
    Обозначим через $D_1, \ldots, D_m$ все триодические и антидиагональные циклы.
    Для любого целочисленного 1-цикла $C$ существует набор чисел $d_1, \ldots, d_m \in \Q$, такой что для любого гомоморфизма $\varphi: H_1(\widetilde K; \Z) \to \Z$ верно
    $\varphi(C) = \sum_{i = 1}^m d_i \varphi(D_i).$
\end{theorem}

Эти две теоремы доказаны в \cite{Al26}, опираясь на результаты статей \cite{Ho07, Dz25, BF09}.


\aronly{
\begin{remark}[эквивалентное определение чисел Ву]
Приведем определение расстановки $u_{f,l}$, отличной от $w_f$, для которого все числа $u_{f,l}(a \times bc)$ будут полуцелыми, а числа Ву (т.~е. значения на целочисленных 1-циклах) --- прежними.

Возьмем любую (неориентированную) прямую $l$ на плоскости, не параллельную ни одной прямой, соединяющей вершины ломаных, являющихся образами ребер графа.
Рассмотрим вектор с началом в $f(a)$ и концом, пробегающим ломаную $f|_{bc}$ в заданном (на ребре $bc$) направлении.
Определим полуцелое число $u_{f,l}(a \times bc)$ как полуразность количеств \emph{прохождений} этим вектором направления прямой $l$ в положительном и в отрицательном направлении (ср. \cite[п. 1.5.4]{Sk}.)

Иными словами, определим отображение $\t f:bc\to S^1$ ребра $bc$ в окружность формулой
$\t f(x):=\dfrac{f(x)-f(a)}{|f(x)-f(a)|}$.
Тогда $u_{f,l}(a \times bc)$ есть полусумма \emph{знаков} \cite[\S8]{Sk20} $\t f$-прообразов двух точек окружности, отвечающих прямой $l$.

Определим $u_{f,l}(e \times a)=u_{f,l}(a \times e)$.  
Построенную расстановку $u_{f,l}$ чисел на ориентированных  ребрах графа $K^{\square \underline2}$ назовем \emph{полуцелочисленной расстановкой вращений}.

Все числа Ву расстановок $u_{f,l}$ и $w_f$ равны ввиду целочисленной версии утверждения \ref{p:stokes} \cite[утверждение 2.4.a]{ABM+}.

\end{remark}



 


\begin{remark}[инвариант Ву]\label{r:defwu}
Приводимое определение аналогично определению в \cite{Ta95}, ср. \cite[\S4]{Sk06}. 

Назовем {\it коциклом} расстановку полуцелых чисел, значение которой на произвольной границе равно нулю. 
Например, расстановка $u_{f, l}$ является коциклом.
А вот другой важный пример коцикла:
{\it кограницей} $\delta v$ вершины $v$ графа $K^{\square\underline2}$ называется расстановка 

$\bullet$ чисел $+1/2$ на ребрах графа $K^{\square\underline2}$, входящих в $v$,

$\bullet$ чисел $-1/2$ на ребрах, выходящих из $v$, и

$\bullet$ нулей на остальных ребрах.

Ясно, что значение кограницы на любом целочисленном 1-цикле равно нулю.


Рассмотрим симметрию (инволюцию) $s$ на $K^{\square\underline2}$, переставляющую компоненты (то есть $s(x,y) =(y,x)$) и отображение, индуцированное этой симметрией, на расстановках.

{\it Симметризованной кограницей} вершины $v$ графа $K^{\square\underline2}$ называется расстановка
$\delta v+\delta sv$.

Расстановки полуцелых чисел на ориентированных ребрах графа $K^{\square\underline2}$ называются \emph{симметрично когомологичными}, если их разность является целочисленной линейной комбинацией симметризованных кограниц вершин этого графа.


\def\Wu{\mathop{\fam0 Wu}}

Обозначим через $H^1_s(\t K;\frac12\Z)$ группу симметричных коциклов  с точностью до симметричной когомологичности (см. замечание \ref{r:delpro-mot}).
\emph{Инвариантом Ву} $\Wu(f):=[u_{f,l}]\in H^1_s(\t K;\frac12\Z)$ называется класс 
полуцелочисленной расстановки вращений.

Корректность этого определения вытекает из утверждения \ref{differ}.b.
Ввиду утверждения \ref{propome0} разность
$\Wu(f)-\Wu(g)$ принимает значения в подгруппе $H^1_s(\t K;\Z)\subset H^1_s(\t K;\frac12\Z)$, полученной аналогично из расстановок \emph{целых} чисел.
Ср. \cite[\S9.3]{Sk}. 
\end{remark}


\begin{proposition}\label{differ} (a) Полуцелочисленные расстановки вращений 
для двух вложений на рисунке~\ref{mcle}.b не являются симметрично когомологичными (ни для какой прямой $l$).
То же для рисунка \ref{mcle}.c.

(b) Для любых прямых $l$ и $l'$ расстановки $u_{f,l}$ и $u_{f,l'}$ симметрично когомологичны.

(c) Для любой вершины $v$ графа $K^{\square\underline2}$ существуют прямые $l$ и $l'$, для которых
\linebreak
$u_{f,l}-u_{f,l'} = \delta v+\delta tv$.
\end{proposition}
 

}

\section{О классификации почти вложений}\label{s:class}

Два (почти) вложения $f,g:K\to\R^2$
называются {\it (почти) изотопными}, если одно можно непрерывно продеформировать в другое так, чтобы в процессе деформации отображение оставалось (почти) вложением.
Вот строгая формулировка этого условия: существует семейство (почти) вложений $f_t:K\to\R^2$, непрерывно зависящее от параметра $t\in[0,1]$, для которого $f_0=f$ и $f_1=g$.


\begin{figure}[h]\centering
\includegraphics[scale=0.35]{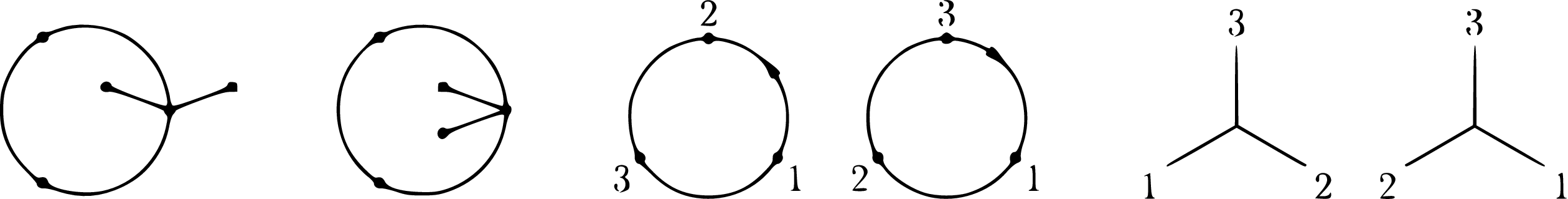}
\caption{Различные вложения
(a) некоторого графа, (b) окружности, и (c) триода}\label{mcle}
\end{figure}

Например, пары вложений на каждом из рисунках \ref{mcle}.abc не изотопны (и даже не почти изотопны).
Это следует из того, что оборотное, циклическое, триодическое числа (а также числа Ву из~\S\ref{s:Wu}) являются инвариантами (почти) вложения 
относительно (почти) изотопии. 

\begin{theorem}\label{t:mcla}
Два вложения связного графа в плоскость изотопны тогда и только тогда, когда их сужения на любой триод и на любой несамопересекающийся цикл изотопны (т.~е. не таковы, как на рисунках~\ref{mcle}.bc).
\end{theorem}


Эту теорему можно сначала доказать для деревьев, а потом свести общий случай к случаю деревьев путем выделения максимального дерева.
Детали технически непросты (как это часто бывает для базовых результатов топологии плоскости).\footnote{Теорема \ref{t:mcla} сформулирована в книге \cite{Wu65} со ссылкой на статью Маклейна-Эдкиссона, которую не удалось найти в указанном там сборнике.
Специалисты по топологической теории графов подтверждают, что эта теорема известна (и верна).
В книге \cite{Wu65} теорема \ref{t:mcla} сформулирована для вложений даже {\it локально связного континуума} (в частности, {\it полиэдра}).}
Аналог теоремы \ref{t:mcla} без утверждения в скобках справедлив для вложений в сферу, тор и другие сферы с ручками.
Доказательство аналогично.

О классификации \emph{погружений} (т.~е. <<локальных вложений>>) графа в плоскость с точностью до \emph{регулярной гомотопии} (т.~е. <<локальной изотопии>>) см. \cite{Pe08, Pe16}.

\begin{remark}\label{r:gene} 
(a) Для вложения графа в плоскость рассмотрим циклические и триодические числа его сужений на всевозможные циклы и триоды в графе (см. замечание \ref{r:noalis}.a).
Ввиду теоремы \ref{t:mcla} \emph{для любого связного графа если все эти числа у двух вложений в плоскость равны, то эти вложения изотопны.} 
(Ср. аналог проблемы~\ref{pm:Wu} для <<вложений>> вместо <<почти вложений>>.)

(b)  Из п. (a) следует, что \emph{если два вложения связного графа в плоскость почти изотопны, то они изотопны}.

(c) Утверждение, аналогичное утверждению из п. (a), для почти изотопий почти вложений неверно (даже если к циклическим и триодическим числам добавить всевозможные числа Ву, в частности, оборотные числа). 
Это вытекает из того, что петля $aba^{-1}b^{-1}$ на плоскости без двух точек не гомотопна постоянной петле.
\end{remark}


\begin{conjecture} \label{c:tree}
    Для любого дерева (или хотя бы для графа $K_{n,1}$) если для двух почти вложений этого дерева равны все триодические числа сужений на всевозможные триоды, то эти почти вложения почти изотопны.
\end{conjecture}


\begin{remark}[аналог проблемы \ref{pm:wind} для вложений]\label{r:embeddings}

По утверждению~\ref{e:k3}.c, $w_f(C,v) \in \{-1, 0, 1\}$ для любой вершины $v \in K$ и любого ориентированного простого цикла $C \in K-v$. 


На данный момент \emph{явного} решения аналога проблемы~\ref{pm:wind} для вложений неизвестно.
Используя следующее описание, можно получить \emph{неявное} описание того, как значения инвариантов меняются при следующих операциях (3)--(0).


Хасслер Уитни описал (см. книги~\cite[\S4.3]{Di97},~\cite[\S2.6]{MT01} и обзор~\cite[p.7, Problem]{Sk05}) простые операции над вложениями графов в сферу, которые позволяют получить любое вложение из любого другого. 
Здесь мы рассматриваем вложения в сферу с точностью до гомеоморфизма сферы.
Нестрого говоря, это описание заключается в следующем.

(3) Планарный 3-связный граф имеет единственное вложение.



\begin{figure}[ht]\centering
\includegraphics[scale=1.]{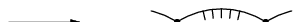}   
\caption{Переворачивание блока}
\label{f:flipping}
\end{figure}

(2) Любое вложение 2-связного графа может быть получено из любого другого <<переворачиванием 3-связной компоненты (блока)>> (рисунок~\ref{f:flipping}).

(1) Любое вложение 1-связного (=связного) графа может быть получено из любого другого <<переворачиванием 2- и 3-связных компонент>>.

(0) Любое вложение 0-связного (=произвольного) графа может быть получено из любого другого <<переворачиванием 1-, 2- и 3-связных компонент>>. 




\end{remark}

\aronly{

\begin{remark}\label{p:wualis}
(a) Для любой прямой $l$ полуцелочисленные расстановки вращений любых (почти) изотопных (почти) вложений $f,f':K\to\R^2$ симметрично когомологичны. 
Поэтому инвариант Ву является инвариантом (почти) изотопии. 

(b) Для любой вершины $v$ графа $K^{\square\underline2}$ существует (почти) вложение $f':K\to\R^2$, (почти) изотопное (почти) вложению $f$, и прямая $l$, для которых $u_{f,l}-u_{f',l} = \delta v+\delta tv$. Это вытекает из утверждения~\ref{differ}.b.
\end{remark}

}

\section{Почти вложения для многомерного случая}\label{s:higher}

\emph{Гиперграфы} --- многомерные аналоги графов: ребрами могут быть не только двухэлементные, но и любые подмножества вершин. 

\begin{remark}\label{r:gap}
(a) Классическая задача топологии, комбинаторики и компьютерной науки --– нахождение критериев (и алгоритмов распознавания) реализуемости гиперграфов в евклидовом пространстве данной размерности $d$. 

\begin{figure}[ht]\centering
\compile{
    \includegraphics[scale=1.2]{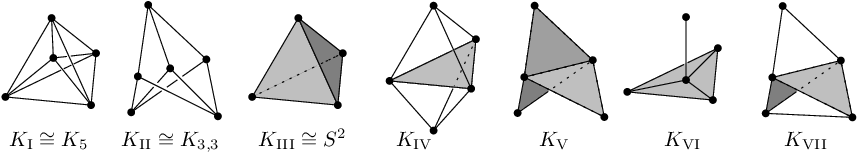}
    }
    \caption{Двумерные гиперграфы, не вложимые в плоскость}
    \label{hj}
\end{figure}

Такой критерий получен классиками топологии (Эгберт ван Кампен, Арнольд Шапиро, Вэньцзюнь Ву, Андре Хэфлигер, Клод Вебер) в 1930-1960-е гг. для $2d\ge3k+3$, где $k$ --- размерность гиперграфа, см. обзор \cite[\S5]{Sk06}. 
Основанный на этом критерии полиномиальный алгоритм распознавания реализуемости получен в 2013 (Мартин Чадек, Мартин Крчал, Лукас Вокжинек), см. обзор \cite[\S3]{Sk18}.

Алгоритмическая неразрешимость для $2d<3k+2$ анонсирована в 2019 году на конференции в Обервольфахе (Марек Филаковский, Ульрих Вагнер и Стефан Жечев).
Дыру в доказательстве нашел \cite{Sk20e} в 2020 году А. Скопенков (она признана авторами).  
Не было доказано, что в многомерном аналоге для \emph{вложений} теоремы \ref{t:k5-e}.b (и примера \cite[Proposition 1.2]{KS20}) некоторый коэффициент зацепления равен $\pm1$. 

В 2020 году Роман Карасев и А. Скопенков показали, что для \emph{почти вложений} этот коэффициент зацепления может принимать любое нечетное значение \cite[Theorem 1.4]{KS20}. 
Их гипотеза об аналогичном результате для графов на плоскости опровергнута в 2023 году Тимуром Гараевым, см. теорему \ref{t:k5-e}.b \cite[Theorem 1]{Ga23}.

(b) В предыдущем пункте (и в других замечаниях этого параграфа), вложения подразумеваются кусочно-линейными (PL), а под реализуемостью понимается существование PL вложения. Очевидно, PL вложимость влечет топологическую (TOP) вложимость (см. определения PL и TOP вложимости в \cite[6.5.1.b]{Sk}).
PL и TOP вложимости $k$-гиперграфов в $\R^d$ равносильны при $d \geqslant k + 3$ \cite{Br72} или при $d = k+1 \in \{2, 3\}$ \cite{Mo77} (см. смежные результаты в \cite[6.5.1.c]{Sk}). 
\end{remark}

\begin{remark}[почти вложения в многомерном случае]
\label{r:ae_complex_mot}

В этом замечании вместо гиперграфов мы говорим о близком понятии \emph{(симплициальных) комплексов}, рассматриваемом в топологии.

Отображение $f:K\to Y$ комплекса $K$ в подмножество $Y\subset\R^d$ называется {\it почти вложением}, если образы несмежных (т.~е. не имеющих общих вершин) граней не пересекаются, т.~е. если $f\sigma\cap f\tau=\emptyset$ для любых несмежных граней $\sigma,\tau$.


(a) Почти вложения комплексов (в многомерные евклидовы пространства) фактически возникли в 1960-х годах у Вебера при изучении вложений комплексов, см. обзор \cite[\S8]{Sk06}. 
Это понятие явно сформулировано в \cite[\S4]{FKT}. 
См. также \cite[\S1, ‘Motivation and background’]{ST17}.
Почти вложения комплексов естественно возникают и применяются также 

$\bullet$ в комбинаторной геометрии 
(теоремы типа Хелли о выпуклых множествах;
см. \cite{Mat97, GPP+}), и 

$\bullet$ в топологической комбинаторике 
(тверберговские теоремы о многократных пересечениях;
см. обзор \cite[\S1]{Sk16}).

Вот недавние работы о почти вложениях комплексов в многообразия: \cite{ST17, KS20, PT19, Al22, DS22, KS21, Sk24}.


(b) \emph{Почему почти вложения, $\Z$- и $\Z_2$-вложения интересны в связи с изучением вложений?} 

См. определения $\Z$- и $\Z_2$-вложений в замечании~\ref{r:zz2} и в \cite[\S 1.1]{Sk24}, \cite[\S\S 6.8, 6.9]{Sk}.

Некоторые доказательства невложимости комплексов в $\R^d$ на самом деле показывают, что эти комплексы не почти вложимы в $\R^d$.
Некоторые из этих доказательств показывают, что эти комплексы даже не $\Z_2$-вложимы в $\R^d$ при $d=2k$.
Это справедливо, например, для

$\bullet$ границы $(d+1)$-симплекса (его не почти вложимость есть топологическая теорема Радона; см. теорему~\ref{t:radon} для $d=2$, обзоры \cite[\S1.1, $(TR_d)$]{Sk16}, \cite[\S2.2 и Theorem 3.1.5]{Sk18}, \cite[Theorem 1.4]{Sk14}),
и

$\bullet$ $k$-комплекса $\Delta_{2k+2}^k$ и $d=2k$ (его не-$\Z_2$-вложимость есть версия теоремы ван Кампена--Флореса; см. теорему~\ref{t:hatuvkfl} для $k=1$, обзоры \cite[\S1.1, $(VKF_{2k})$]{Sk16}, \cite[\S1.4 и Theorem 3.1.6]{Sk18}).


С другой стороны, некоторые построения вложений содержат построения почти или $\Z$-вложений как удобный промежуточный шаг, позволяющий структурировать доказательство и описать связь с известными результатами и методами.





О связях почти вложимости и $\Z$-, $\Z_2$-вложимости с вложимостью комплексов см. теоремы~\ref{t:vkswum},~\ref{r:pltop} и \cite[Theorem 1.2.1a]{Sk24}, \cite[\S6.8, \S6.9 and Remarks 8.2.5.ab]{Sk}.


(d) Аналог замечания~\ref{r:ae_graph_mot}.c справедлив для почти вложений комплексов в многообразия.

(e) Родственное, но другое понятие изучалось под названием почти вложение в \cite{Sk05i, Sk06c, CRS, CRS'}.


(f) Изучение \emph{вложений} $k$-комплексов в $2k$-многообразия при $k>1$ аналогично изучению \emph{$\Z_2$-вложений} (а не вложений) графов в поверхности.
Действительно, неравенство Эйлера $V-E+F\ge\chi(M)$ не доказано, а может быть и неверно, для $\Z_2$-вложений графов в поверхность $M$. 
Для изучения $\Z_2$-вложений графов в поверхность вместо неравенства Эйлера применяется форма пересечений поверхности. 
Критерии вложимости $k$-комплексов в $2k$-многообразие при $k>1$ также получены на языке формы пересечений, см. \cite{PT19, Sk24}.
Об идее, в некотором смысле заменяющей неравенство Эйлера, и о реализации этой идеи, см. \cite{Ka91}, \cite{Ad18}, \cite[\S6]{DS22}.



\end{remark}




\begin{theorem}[Вебер]\label{t:vkswum}
Если $d\ge\frac{3(k+1)}2$ и $k$-комплекс почти вложим в $\R^d$, то этот комплекс PL вложим в $\R^d$.
\end{theorem}


См. обзор \cite[\S8]{Sk06}; ср. \cite[Theorem 3.2.6]{Sk18}.


\begin{theorem}\label{r:pltop} 
Для любых $d,k$, таких что $k+2\le d\le\frac{3k}2+1$ (в частности, для $d=2k=4$) существует $k$-комплекс почти вложимый в $\R^d$, но не PL вложимый в $\R^d$.
\end{theorem}



См. \cite[Пример на стр. 338]{SSS}; ср. \cite{SS92};  
изложение идеи доказательства из \cite{SSS} приведено в обзоре \cite[\S7]{Sk06}.


Неизвестно, существует ли 2-комплекс почти вложимый в $\R^3$, но не вложимый в $\R^3$. 
(Несложно и известно, что конус над $K_5$ или над $K_{3,3}$ не вложим в $\R^3$, ср. \cite[\S4.4, Proof of Proposition 4.1.a]{Sk14}.
Конус над $K_5$ не почти вложим в $\R^3$ по теореме Конвея--Гордона--Закса \cite[Предложение 1.2.9.a]{Sk24'}.
Мы выдвигаем гипотезу, что конус над $K_{3,3}$ не почти вложим в $\R^3$.)



Родственный, но другой пример "--- это 2-комплекс, $\Z$-вложимый, но не почти вложимый в $\R^4$ \cite[\S3.2, \S3.3, \S4]{FKT}; 
доказательство, не использующее теорему Столлингса о центральных рядах групп, см. в \cite[Theorem 1.6]{AMSW} и \cite{Al22}.



\begin{theorem}[\cite{ST17}]
\label{t:nphhae}
Для любых фиксированных $d,k\ge2$, таких что $d=\frac{3k}2+1$, алгоритмическая задача распознавания почти вложимости $k$-комплексов в $\R^d$ является \textbf{NP}-трудной.
\end{theorem}


По-видимому, теорема верна и для $k+2\le d < \frac{3k}2+1$ \cite{Al22}.

Аналогичный результат для вложений справедлив и принадлежит Иржи Матушеку, Мартину Танцеру и Ульриху Вагнеру (см. \cite{MTW} и обзор {\cite[Theorem 3.2.8]{Sk18}}).

\begin{problem}\label{c:links}
Для каких $k,d$ существует алгоритм распознавания почти вложимости $k$-комплексов в $\R^d$?
\end{problem}


Для $d\ge3(k+1)/2$ существует (полиномиальный) алгоритм, поскольку для $d\ge3(k+1)/2$ и $k$-комплексов в $\R^d$

$\bullet$ почти вложимость эквивалентна PL вложимости  (теорема~\ref{t:vkswum}); и

$\bullet$ существует (полиномиальный) алгоритм распознавания PL вложимости (см. второй абзац в замечании~\ref{r:gap}).




{\it Книги, обзоры и методические статьи в этом списке помечены звездочками.}
 
\end{document}